\documentclass[12pt]{article}
\usepackage{amsmath,amsfonts,amssymb,amsthm,mathrsfs,latexsym,relsize,graphicx,color,tikz}
\usepackage[colorlinks,citecolor=red,pagebackref,hypertexnames=false]{hyperref}

\usepackage{tocloft}
\newtheorem{theorem}{Theorem}
\newtheorem{lemma}[theorem]{Lemma}
\newtheorem{proposition}[theorem]{Proposition}
\newtheorem{definition}[theorem]{Definition}
\newtheorem{corollary}[theorem]{Corollary}

\newtheorem{example}[theorem]{Example}

\newtheorem{remark}[theorem]{Remark}

\newcommand{\R}{\mathbb{R}}
\newcommand{\T}{\mathbb{T}}
\newcommand{\C}{\mathbb{C}}
\newcommand{\Z}{\mathbb{Z}}

\newcommand{\RR}{\mathbf{R}}
\newcommand{\PP}{\mathcal{P}}
\newcommand{\Hil}{\mathscr{H}}
\newcommand{\id}{\textnormal{e}}
\newcommand{\dg}{d_{\wh{G}}}
\newcommand{\dr}{d_{\wh{\RR}}}
\newcommand{\dalpha}{\wh{\alpha}}
\newcommand{\one}[1]{\mathlarger{\chi}_{\raisebox{-.5ex}{$\mathsmaller{#1}$}}}
\newcommand{\wh}[1]{\widehat{#1}}
\newcommand{\ol}[1]{\overline{#1}}
\newcommand{\ul}[1]{\underline{#1}}

\begin{document}
\title{Calder\'on-type inequalities for affine frames}

\author{Davide Barbieri, Eugenio Hern\'andez, Azita Mayeli}

\date{\today}
\maketitle

\vspace{-3.5ex}
\begin{abstract}
We prove sharp upper and lower bounds for generalized Calder\'on's sums associated to frames on LCA groups generated by affine actions of cocompact subgroup translations and general measurable families of automorphisms. 
The proof makes use of techniques of analysis on metric spaces, and relies on a counting estimate of lattice points inside metric balls. We will deduce as special cases Calder\'on-type inequalities for families of expanding automorphisms as well as for LCA-Gabor systems.
\end{abstract}

\vspace{-2.5ex}

{
\hypersetup{linkcolor=black}
\setcounter{tocdepth}{2}
\tableofcontents
\clearpage
}

\section{Introduction}

If $\psi \in L^2(\R)$ and $a \in \R_+ \setminus \{1\}$, several well-known conditions for discrete wavelets systems are given in terms of the Calder\'on's sum \[\mathcal{C}_\psi(\xi) = \sum_{j \in Z} |\wh{\psi}(a^j\xi)|^2.\] A necessary condition for an affine system $\Psi = \{a^{\frac{j}{2}}\psi(a^j \cdot - k)\}_{j,k \in Z}$ to be an orthonormal basis of $L^2(\R)$ (see \cite{Gripenberg}, \cite{Wang}) is
\[
\mathcal{C}_\psi(\xi) = 1 \quad \textnormal{a.e.} \ \xi \in \R
\]
which is also sufficient to prove that the system is complete once it is known to be orthonormal (see \cite{Rzeszotnik}, \cite{Bownik}). On the other hand, a necessary condition for $\Psi$ to be a frame with constants $A$ and $B$ is given by (see \cite{ChuiShi})
\begin{equation}\label{eq:Calderon1}
A \leq \mathcal{C}_\psi(\xi) \leq B \quad \textnormal{a.e.} \ \xi \in \R .
\end{equation}
A consequence of (\ref{eq:Calderon1}) is that the so-called coaffine systems $\{a^{\frac{j}{2}}\psi(a^j (\cdot - k))\}_{j,k \in Z}$ can not form a frame of $L^2(\R)$ (see \cite{GLWW}).
Let us also recall that the Calder\'on's sum is actually related to the well-known Calder\'on's admissibility condition characterizing continuous wavelets (see e.g. \cite{WeissWilson}).

The results on orthonormal wavelets on $\R$ have been extended to several degrees of generality, including characterizations of tight frames for generalized shift-invariant systems on $\R^n$ \cite{HLW}, and on LCA groups \cite{KutyniokLabate}, and more recently these results were obtained under weaker hypotheses in a setting that unifies discrete and continuous systems \cite{JakobsenLemvigTranslations}. The setting of general frames is much less studied, but we note that very recently, a closely related result to the one discussed in the present paper has been obtained in \cite[Theorem 6.2]{FuhrLemvig} with completely different techniques.

\

In this paper we consider frames in subspaces of $L^2(G)$, where $G$ is an LCA group, generated by translations by a cocompact subgroup $\Gamma$ and by a family of automorphisms, endowed with a Borel measure, that formally replaces dilations. The generality of the setting also allows us to consider modulations, hence including Gabor-type systems, after a proper choice of automorphisms and subspaces of $L^2(G)$, see Section \ref{sec:Gabor}. Our main requirements on these objects are that the dual group $\wh{G}$ has an invariant metric that satisfies the Lebesgue's differentiation theorem, and that the family of automorphisms be bi-Lipschitz, with measurable upper and lower constants, with respect to such an invariant distance.

The main result of this paper is Theorem \ref{theo:main}, which proves that affine frames in this setting satisfy inequalities that generalize (\ref{eq:Calderon1}). The proof relies on a variant of a classical strategy which consists of testing the frame condition on a family of functions whose Fourier transforms define an approximation of the identity, and then obtain the inequalities in the limit by Lebesgue's differentiation theorem. This strategy can be performed whenever two crucial hypotheses are met.

\

The first hypothesis is a counting estimate, that requires the number of points of the discrete annihilator $\Gamma^\bot$ inside the dilation of small $\wh{G}$ balls to grow at most as much as the jacobian of the dilation. This condition, that we call Property X, see Definition \ref{def:PropertyX}, is similar to the lattice counting estimate studied recently in \cite{BownikLemvig} for automorphisms given by powers of a matrix. The lattice counting estimate appeared in previous works, notably \cite{HLW}, \cite{GuoLabate}, \cite{KutyniokLabate}, and in \cite{BownikLemvig} the authors prove that the estimate is equivalent to the automorphisms to be expanding on a subspace (and not contracting on the complementary). In the present paper, in Definition \ref{def:expanding}, we introduce a general notion of expansiveness, which generalizes that of expanding on a subspace to all bi-Lipschitz families of automorphisms, see Proposition \ref{prop:expandingsubspace}. We then prove in Theorem \ref{theo:expandingX} that this notion of expansiveness is actually sufficient to obtain Property X. The proof relies on a technical counting estimate, given by Lemma \ref{lem:counting}, which holds for any discrete subgroup of an LCA group and that is actually optimal, as shown in Lemma \ref{lem:optimal}. However, we can prove Property X also in nonexpanding settings. In Section \ref{sec:examples} we indeed obtain Property X for shearlets, which are systems of wavelets with composite dilations that are not expanding, and in Lemma \ref{lem:GaborX} we can obtain it for Gabor systems, after an appropriate adjustment of the involved Hilbert space.

\

The second hypothesis which is required for the proof of Theorem \ref{theo:main} is a local integrability condition for the mother wavelet. This requirement appears naturally when applying Lebesgue's differentiation, since it provides the local integrability of one of the terms of the estimate, actually a remainder which does vanish. Essentially the same hypothesis but in a slightly different setting, that of generalized shift-invariant systems on $\R$, was thoroughly discussed in the recent paper \cite{ChristensenHasannasabLemvig}, and it was related to the so-called LIC and $\alpha$-LIC conditions introduced respectively in \cite{HLW} and in \cite{JakobsenLemvigTranslations}. Here, in Theorem \ref{theo:sufficient}, we provide sufficient conditions on the family of automorphisms, in terms of a slightly stronger notion of expansiveness, that makes this local integrability condition to hold automatically for all mother wavelets. This condition is not necessary, since for example in the nonexpanding case of Gabor systems we get local integrability for free. However, it may be interesting to observe that essentially the same idea of the proof was used in \cite{ChristensenHasannasabLemvig} to obtain this local integrability for all wavelets on $\R$ (which satisfy a special case of the mentioned stronger expansiveness), and in our opinion the core of this argument can be found also in \cite{ChuiShi}.

\

Once Theorem \ref{theo:main} is established, we can then obtain Calder\'on's inequalities for expanding automorphisms in  Corollary \ref{cor:expandingCalderon}, as well as for Gabor systems in Theorem \ref{theo:GaborCalderon}, by showing that Property X holds.

\

We now compare our results with similar ones that have appeared in the literature on the subject.

In \cite{JakobsenLemvigTranslations} the authors work in a different setting and prove the upper bound of the Calder\'on sum for Bessel sequences, or treat the case of dual frames. We treat the case of general frames to find the lower bound, and allow for continuous families of automorphism.

In \cite{JakobsenLemvigGabor} the authors consider Gabor systems, and in this setting their assumptions are more general than ours, while our results apply to other systems as well. They prove the upper and lower bound for the Calder\'on sum for Gabor systems. Our Theorem \ref{theo:GaborCalderon} proves the same result in a setting that is less general with respect to Gabor systems, but using completely different techniques that apply to more general systems.

In \cite{BownikLemvig} the authors work in $\R^n$, where they obtain the Calder\'on condition for Parseval frames and dual frames. They study several issues concerning a counting estimate similar to Property X, and they prove that it is equivalent to the expanding on a subspace condition. We treat general second countable LCA groups and general families of automorphisms, with a definition of expanding that implies Property X, but is not equivalent although it includes expanding on a subspace matrices.

In \cite{HLW} the authors work in $\R^n$, they prove the upper bound for the Calder\'on sum for Bessel sequences, and they obtain the Calder\'on condition for Parseval frames or dual frames.

\vspace{3ex}
\paragraph{Acknowledgements:}
This project has received funding from the European Union’s Horizon 2020 research and innovation programme under the Marie Skłodowska-Curie grant agreement No 777822. D. Barbieri and E. Hern\'andez were supported by Grants MTM2016-76566-P (MINECO, Spain).  A. Mayeli was supported by PSC-CUNY grant 60623-00 48.\\
The authors would like to thank Jakob Lemvig for interesting and helpful comments on a previous version of this paper.

\paragraph{Keywords:} Frames in LCA groups; Calder\'on condition for frames; Gabor systems.
\newpage

\section{Preliminaries}\label{sec:preliminaries}

Let $G$ be a locally compact and second countable abelian group, and let us denote with $+$ its composition law. Let
\[
\wh{f}(\xi) = \int_G f(x) \ol{\langle \xi,x\rangle} d\mu_G(x)
\]
be the $G$-Fourier transform, where we denote by $\langle \xi,x\rangle \in \T$ the pairing with a character $\xi \in \wh{G}$.

Let $\Gamma$ be a cocompact closed subgroup of $G$. Let us fix a Haar measure $\mu_G$ on $G$ and, since $G/\Gamma$ is compact, let us fix the normalized measure $\kappa$ on $G/\Gamma$, i.e. such that $\kappa(G/\Gamma) = 1$.

Let us then set a Haar measure $\mu_\Gamma$ on $\Gamma$ in such a way that the corresponding Weil's formula holds with constant 1:
\[
\int_G f(x) d\mu_G(x) = \int_\Gamma \big(\int_{G/\Gamma} f(y+\gamma) d\kappa(y)\big) d\mu_\Gamma(\gamma).
\]

Let $\Gamma^\perp = \{\lambda \in \wh{G} \, : \, \langle \lambda,\gamma\rangle = 1 \ \forall \, \gamma \in \Gamma\}$ denote the annihilator of $\Gamma$, which is a closed subgroup of $\wh{G}$. Observe that, since $G/\Gamma$ is compact, then $\Gamma^\perp \approx \wh{(G/\Gamma)}$ is discrete, so it is countable because $\wh{G}$ is second countable.

We will make use of the following standard result relating the multiplicative constants of Haar measures with Plancherel theorems on subgroups. It was originally given by Weil \cite{Weil}, a proof can be found also in \cite[(31.46), (c)]{HewittRossII}, and the argument to obtain it is reported in Appendix \ref{sec:Weil}.
\begin{lemma}\label{lem:Weil}
Let $(G,\mu_G)$ and $(\Gamma,\mu_\Gamma)$ be as above, and let $\nu_{\wh{G}}$ be the Haar measure on $\wh{G}$ that makes the $G$-Fourier transform a unitary operator from $L^2(G)$ to $L^2(\wh{G})$, i.e.
\[
\int_G |f(x)|^2 d\mu_G(x) = \int_{\wh{G}} |\wh{f}(\xi)|^2 d\nu_{\wh{G}}(\xi) \qquad \forall \ f \in L^2(G) .
\]
Let $\Omega \subset \wh{G}$ be any $\nu_{\wh{G}}$-measurable section of $\wh{G}/\Gamma^\perp$, i.e. a fundamental set.

Then $\nu_{\wh{G}}(\Omega) > 0$ and the following identities hold
\begin{align*}
i. & \int_{\wh{G}} \phi(\xi) d\nu_{\wh{G}}(\xi) = \sum_{\lambda \in \Gamma^\perp} \int_\Omega \phi(\xi + \lambda) d\nu_{\wh{G}}(\xi) \qquad \forall \ \phi \in L^1(\wh{G}) .\\
ii. & \int_\Omega |\phi(\xi)|^2 d\nu_{\wh{G}}(\xi) = \int_\Gamma |\int_\Omega \phi(\xi) \langle \gamma, \xi \rangle d\nu_{\wh{G}}(\xi)|^2 d\mu_\Gamma(\gamma) \qquad \forall \ \phi \in L^2(\Omega) .
\end{align*}
\end{lemma}

\subsection{Counting estimates on metric abelian groups}

Since $\wh{G}$ is locally compact and second countable it is metrizable, its metric can be chosen to be invariant under the group action, and it is complete\footnote{Actually it is proper: closed balls are compact.} (see \cite[Theorem (8.3)]{HewittRossI}, \cite{Struble}). Let us denote with $\dg : \wh{G} \times \wh{G} \to \R^+$ one of these metrics and with
\[
B(\xi_0,r) = \{\xi \in \wh{G} \, : \, \dg(\xi,\xi_0)<r\}
\]
a $\dg$-metric ball in $\wh{G}$ of radius $r > 0$ and center $\xi_0$. Note that, denoting by $\id$ the identity element of $\wh{G}$, by the invariance of the metric we have
\[
B(\xi_0,r) = B(\id,r) + \xi_0 .
\]

In the next sections we will need a counting estimate for the number of the $\Gamma^\perp$ lattice points that lie inside metric balls deformed by an automorphism $\dalpha \in \textnormal{Aut}(\wh{G})$. We will deduce it in some relevant settings from a basic counting lemma that we now present. In order to obtain it, let us introduce the following notation: given $r > 0$, denote by
\[
(\Gamma^\perp)_{\dalpha}^r = \bigcup_{\lambda \in \Gamma^\perp} \big(\dalpha B(\id,r) + \lambda\big) .
\]
Also, for $\Omega$ a fundamental domain for $\Gamma^\perp$ in $\wh{G}$, let us define
\begin{equation}\label{eq:neighborhood}
\Omega_{\dalpha}^r = \Omega \cap (\Gamma^\perp)_{\dalpha}^r.
\end{equation}
Note that $\Omega_{\dalpha}^r$ has a finite measure since, using $i.$, Lemma \ref{lem:Weil}, we get
\begin{align*}
\nu_{\wh{G}}&(\Omega_{\dalpha}^r) = \int_{\Omega \cap (\Gamma^\perp)_{\dalpha}^r} d\nu_{\wh{G}}(\xi) \leq \sum_{\lambda \in \Gamma^\perp} \int_{\Omega \cap \big(\dalpha B(\id,r) + \lambda\big)} d\nu_{\wh{G}}(\xi)\\
& = \sum_{\lambda \in \Gamma^\perp} \int_{\Omega + \lambda} \one{\dalpha B(\id,r)}(\xi) d\nu_{\wh{G}}(\xi) = \int_{\wh{G}} \one{\dalpha B(\id,r)}(\xi) d\nu_{\wh{G}}(\xi) = \nu_{\wh{G}}(\dalpha B(\id,r)).
\end{align*}
We can now state our counting lemma in general terms as follows.
\begin{lemma}\label{lem:counting}
Let $\wh{G}$ be a second countable LCA group with Haar measure $\nu_{\wh{G}}$, let $\Gamma^\perp$ be a discrete subgroup of $\wh{G}$, and let $\dalpha \in \textnormal{Aut}(\wh{G})$. Then
\[
\sharp \Big(\Gamma^\perp \cap \dalpha B(\id,r)\Big) \leq \frac{1}{\nu_{\wh{G}}(\Omega_{\dalpha}^r)}\nu_{\wh{G}}(\dalpha B(\id,2r)) \quad \forall \ r > 0
\]
where $\Omega_{\dalpha}^r$ is as in (\ref{eq:neighborhood}) for any $\Omega \subset \wh{G}$ a $\nu_{\wh{G}}$-measurable section of $\wh{G}/\Gamma^\perp$.
\end{lemma}

\begin{proof}
For all $\xi \in \wh{G}$ and $r > 0$, let us denote by
\[
S_{\dalpha}^r(\xi) = \sharp \Big(\Gamma^\perp \cap (\dalpha B(\id,r) - \xi)\Big) = \sum_{\lambda \in \Gamma^\perp} \one{\dalpha B(\id,r) - \xi}(\lambda) = \sum_{\lambda \in \Gamma^\perp} \one{\dalpha B(\id,r) + \lambda}(\xi)
\]
and observe that supp$\, S_{\dalpha}^r \subset (\Gamma^\perp)_{\dalpha}^r$.
By $i.$, Lemma \ref{lem:Weil} we then have
\begin{align}\label{eq:basiccounting}
\nu_{\wh{G}}(\dalpha(& B(\id,r))) = \int_{\wh{G}} \one{\dalpha(B(\id,r))}(\xi) d\nu_{\wh{G}}(\xi) = \sum_{\lambda \in \Gamma^\perp} \int_{\Omega} \one{\dalpha(B(\id,r))}(\xi + \lambda) d\nu_{\wh{G}}(\xi)\nonumber\\
& = \int_{\Omega} S_{\dalpha}^r(\xi) d\nu_{\wh{G}}(\xi) = \int_{\Omega_{\dalpha}^r} S_{\dalpha}^r(\xi) d\nu_{\wh{G}}(\xi) \geq \int_{\Omega_{\dalpha}^{\frac{r}{2}}} S_{\dalpha}^r(\xi) d\nu_{\wh{G}}(\xi).
\end{align}
Now, by the definition of $(\Gamma^\perp)_{\dalpha}^{\frac{r}{2}}$, for $\xi \in \Omega_{\dalpha}^{\frac{r}{2}}$ there exists at least one $\lambda_\xi \in \Gamma^\perp$ such that
\begin{equation}\label{eq:whereisxi}
\xi \, \in \, \dalpha B(\id,{\frac{r}{2}}) + \lambda_\xi = \dalpha B(\dalpha^{\, -1}(\lambda_\xi),{\frac{r}{2}}) .
\end{equation}
We claim that, when $\xi \in \Omega_{\dalpha}^{\frac{r}{2}}$, and for $\lambda_\xi$ as above, it holds
\begin{equation}\label{eq:trickclaim}
\dalpha B(\id,{\frac{r}{2}}) - \lambda_\xi \subset \dalpha B(\id,r) - \xi .
\end{equation}
Claim (\ref{eq:trickclaim}) implies that, when $\xi \in \Omega_{\dalpha}^{\frac{r}{2}}$, we can estimate from below the number of points in the $\dalpha$ image of a ball translated by $\xi$ with
\begin{align*}
S_{\dalpha}^r(\xi) & = \sharp \Big(\Gamma^\perp \cap (\dalpha B(\id,r) - \xi)\Big) \geq \sharp \Big(\Gamma^\perp \cap (\dalpha B(\id,{\frac{r}{2}}) - \lambda_\xi)\Big)\\
& = \sharp \Big(\Gamma^\perp \cap \dalpha B(\id,{\frac{r}{2}})\Big) = S_{\dalpha}^{\frac{r}{2}}(\id).
\end{align*}
By inserting this into (\ref{eq:basiccounting}) we then get the desired estimate
\[
\nu_{\wh{G}}(\dalpha B(\id,r)) \geq \nu_{\wh{G}}(\Omega_{\dalpha}^{\frac{r}{2}}) S_{\dalpha}^{\frac{r}{2}}(\id).
\]
In order to prove (\ref{eq:trickclaim}), let $\zeta = \dalpha(\eta) \in \dalpha B(\id,{\frac{r}{2}}) - \lambda_\xi$. Then
\begin{align*}
\dg(\eta,-\dalpha^{\, -1}(\xi)) & \leq \dg(\eta,-\dalpha^{\, -1}(\lambda_\xi)) + \dg(-\dalpha^{\, -1}(\lambda_\xi),-\dalpha^{\, -1}(\xi)) \\
& < \frac{r}{2} + \dg(\dalpha^{\, -1}(\lambda_\xi),\dalpha^{\, -1}(\xi)) < \frac{r}{2} + \frac{r}{2}
\end{align*}
where the last inequality is due to (\ref{eq:whereisxi}).
This implies that $\eta \in B(-\dalpha^{\, -1}(\xi),r)$, so $\zeta \in \dalpha B(\id,r) - \xi$.
\end{proof}

\subsection{Measurable families of bi-Lipschitz automorphisms}\label{subsec:automorph}
Let $G$ be a locally compact and second countable abelian group.
Let $(H,\varsigma)$ be a measure space, and let $\alpha : H \to \textnormal{Aut}(G)$ be a measurable family of $G$-automorphisms, i.e. assume that
\[
(x,h) \mapsto \alpha_h(x) \ \textnormal{is measurable in} \ (G,\mu_G) \times (H,\varsigma) .
\]
Let $\delta : H \to \R^+$ be the $\varsigma$-measurable map satisfying (see \cite[(15.26)]{HewittRossI})
\[
\int_G f(\alpha_h(x)) d\mu_G(x) = \delta(h) \int_G f(x) d\mu_G(x) \qquad \forall \ f \in L^1(G).
\]
Let us define the dual action $\dalpha : H \to \textnormal{Aut}(\wh{G})$ by\footnote{When $H$ is a group and $\alpha$ is a homomorphism, this too is a homomorphism, since \[\langle \dalpha_{h_1}(\dalpha_{h_2}(\xi)),x\rangle = \langle \dalpha_{h_2}(\xi),\alpha_{h_1}^{-1}(x)\rangle = \langle \xi,\alpha_{h_1 h_2}^{-1}(x)\rangle = \langle \dalpha_{h_1h_2}(\xi),x\rangle.\]}
\[
\dalpha_h(\xi) \, : \, \langle \dalpha_h(\xi),x\rangle = \langle \xi,\alpha_h^{-1}(x)\rangle .
\]
It is well-known that $\dalpha$-changes of variables on $\wh{G}$ read as follows.
\begin{lemma}\label{lem:changeofvar}
For all $\phi \in L^1(\wh{G})$ it holds
\[
\int_{\wh{G}} \phi(\dalpha_h(\xi)) d\nu_{\wh{G}}(\xi) = \delta(h)^{-1} \int_{\wh{G}} \phi(\xi) d\nu_{\wh{G}}(\xi) .
\]
\end{lemma}
\begin{proof}[Sketch of the proof]
Since $\mu_G$ and $\nu_{\wh{G}}$ are such that the $G$-Fourier transform is unitary, for all $f \in L^2(G)$ we have
\begin{align*}
& \int_{\wh{G}} |\wh{f}(\xi)|^2 d\nu_{\wh{G}}(\xi) = \int_G |f(x)|^2 d\mu_G(x) = \delta(h)^{-1}\int_G |f(\alpha_h(x))|^2 d\mu_G(x)\\
& = \delta(h)^{-1}\int_{\wh{G}} |\int_G f(\alpha_h(x)) \ol{\langle \xi, x \rangle} d\mu_G(x)|^2 d\nu_{\wh{G}}(\xi)\\
& = \delta(h) \int_{\wh{G}} |\int_G f(x) \ol{\langle \xi, \alpha_h^{-1}(x) \rangle} d\mu_G(x)|^2 d\nu_{\wh{G}}(\xi) = \delta(h) \int_{\wh{G}} |\wh{f}(\dalpha_h(\xi))|^2 d\nu_{\wh{G}}(\xi)
\end{align*}
which proves the desired identity for all $\phi \in L^1(\wh{G})$ with values in $\R^+$.
\end{proof}
\begin{remark}\label{rem:measureballdilations}
Note that, by Lemma \ref{lem:changeofvar}, we have
\[
\nu_{\wh{G}} (\dalpha_hB(\xi_0,r)) = \delta(h) \nu_{\wh{G}} (B(\xi_0,r)) \quad \forall \, \xi_0 \in \wh{G}, \, \forall \, r > 0, \, \forall \, h \in H.
\]
\end{remark}

In order to control the deformation on $\wh{G}$ due to the action of $\dalpha$ we will use the notion of bi-Lipschitz maps. If $(E,d)$ is a metric space, a bijection $f : E \to E$ is called bi-Lipschitz if there exist $0 < \ell \leq L < \infty$ such that
\begin{equation}\label{eq:bilip}
\ell d(x,y) \leq d(f(x),f(y)) \leq L d(x,y) \quad \forall \ x,y \in E.
\end{equation}
For bi-Lipschitz maps, one has the following known result.
\begin{lemma}\label{lem:balls}
Let $f : E \to E$ be a bi-Lipschitz map on the metric space $(E,d)$ and for $x \in E$, $r > 0$ let us denote by $B(x,r) = \{y \in E \, : \, d(y,x) < r\}$. Then
\[
B(f(x),\ell r) \subset f(B(x,r)) \subset B(f(x),L r) \quad \forall \ x \in E \, , \ \forall \ r > 0.
\]
\end{lemma}
\begin{proof}
To prove the inclusion from above, let $y \in f(B(x,r))$, i.e. $y = f(z)$ for some $z \in E$ such that $d(z,x) < r$. Then, by the right hand side inequality in (\ref{eq:bilip}), we have
\[
d(y,f(x)) = d(f(z),f(x)) \leq L d(z,r) < L r
\]
so $y \in B(f(x),L r)$. To prove the inclusion from below, let $y \in B(f(x),\ell r)$, and let $z = f^{-1}(y)$. Then, by the left hand side inequality in (\ref{eq:bilip}), we have
\[
\ell d(z,x) \leq d(f(z),f(x)) = d(y,f(x)) < \ell r
\]
so $z \in B(x,r)$, i.e. $y = f(z) \in f(B(x,r))$.
\end{proof}

The class of families of automorphisms that we will consider is that satisfying the following definition.
\begin{definition}
We say that the map $\dalpha$ is \emph{measurably bi-Lipschitz} if there exist two $\varsigma$-measurable functions $\ell : H \to \R^+$ and $L : H \to \R^+$ satisfying
\begin{equation}\label{eq:metric}
\ell(h) \dg(\xi,\id) \leq \dg(\dalpha_h(\xi),\id) \leq L(h) \dg(\xi,\id) \quad \forall \ \xi \in \wh{G} .
\end{equation}
\end{definition}
For such families, by Lemma \ref{lem:balls}, we have that
\begin{equation}\label{eq:balls}
B(\dalpha_h (\xi_0), \ell(h) r) \subset \dalpha_h B(\xi_0,r) \subset B(\dalpha_h (\xi_0), L(h) r) \quad \forall \ \xi_0 \in \wh{G} , \ \forall \ h \in H.
\end{equation}


\subsection{The frame condition}
Let $G$ be a locally compact and second countable abelian group, and let us denote with $+$ its composition law. Let $\Gamma$ be a cocompact subgroup of $G$.
Let $T : \Gamma \to \mathcal{U}(L^2(G))$ be the unitary representation of $\Gamma$-translations on $L^2(G)$, i.e. $T(\gamma) f(x) = f(x - \gamma)$, and let $D : H \to \mathcal{U}(L^2(G))$ be the unitary operator-valued map given by
\[
D(h)f(x) = \delta(h)^{-\frac12}f(\alpha_h(x))
\]
where $H$ and $\alpha$ are as in \S \ref{subsec:automorph}.
Let us also denote with $\wh{D} : H \to \mathcal{U}(L^2(\wh{G}))$ the unitary operator-valued map obtained by conjugation with the Fourier transform, i.e. such that $\wh{D}(h)\wh{f} = \wh{D(h)f}$. It reads explicitly
\begin{equation}\label{eq:dualdilations}
\wh{D}(h) \wh{f}(\xi) = \delta(h)^{\frac12} \wh{f}(\dalpha_h(\xi))
\end{equation}
since, by definition of $\delta$ and $\dalpha$, we have
\begin{align*}
\wh{D(h) f}(\xi) & = \int_G D(h)f(x) \ol{\langle \xi, x \rangle} d\mu_G(x) = \delta(h)^{-\frac12} \int_G f(\alpha_{h}(x)) \ol{\langle \xi, x \rangle} d\mu_G(x)\\
& = \delta(h)^{\frac12} \int_G f(x) \ol{\langle \xi, \alpha_{h}^{-1}(x) \rangle} d\mu_G(x) = \delta(h)^{\frac12} \wh{f}(\dalpha_h(\xi)).
\end{align*}

The two main notions that we will relate in the next sections are the following generalized definitions of affine frame and of Calder\'on's sum.

\begin{definition}\label{def:frames}
With the above notations, let $\Hil$ be a closed subspace of $L^2(G)$ with the induced norm. For $\psi \in L^2(G)$, we say that the affine system
\[
\mathcal{A}_{\Gamma,H}(\psi) = \Big\{D(h)T(\gamma)\psi \, : \, \gamma \in \Gamma \, , \ h \in H\Big\}
\]
is a $\varsigma$-frame of $\Hil$ with constants $0 < A \leq B < \infty$ if $\mathcal{A}_{\Gamma,H}(\psi) \subset \Hil$ and
\begin{equation}\label{eq:frame}
A \|f\|_{L^2(G)}^2 \leq \int_H \Big(\int_\Gamma |\langle f, D(h)T(\gamma)\psi\rangle_{L^2(G)}|^2  d\mu_\Gamma(\gamma) \Big) d\varsigma(h) \leq B \|f\|_{L^2(G)}^2
\end{equation}
holds for all $f \in \Hil$. If only the right inequality holds for all $f \in \Hil$, we say that $\mathcal{A}_{\Gamma,H}(\psi)$ is a $\varsigma$-Bessel system for $\Hil$.
\end{definition}

\begin{definition}
For $\psi \in L^2(G)$, we call $\varsigma$-Calder\'on sum for $\mathcal{A}_{\Gamma,H}(\psi)$ the expression
\[
\mathcal{C}_\psi(\xi) = \int_H |\wh{\psi}(\dalpha_h(\xi))|^2 d\varsigma(h) .
\]
\end{definition}

As for standard affine systems, with the next lemma we provide a trivial but key equation which has a fundamental role in the following computations.

\begin{lemma}\label{lem:fourier}
If $\mathcal{A}_{\Gamma,H}(\psi)$ is a $\varsigma$-Bessel system for $\Hil \subset L^2(G)$, then for all $f \in \Hil$ the central term in (\ref{eq:frame}) reads
\begin{align*}
\mathcal{I}_{\Gamma,H}(f,\psi) & = \int_H \Big(\int_\Gamma |\langle f, D(h)T(\gamma)\psi\rangle_{L^2(G)}|^2  d\mu_\Gamma(\gamma) \Big) d\varsigma(h)\\
& = \int_H \delta(h)^{-1} \Big(\int_\Omega |\sum_{\lambda \in \Gamma^\perp} \wh{f}(\dalpha_h^{\,-1}(\xi + \lambda)) \ol{\wh{\psi}(\xi + \lambda)}|^2  d\nu_{\wh{G}}(\xi) \Big) d\varsigma(h)
\end{align*}
where $\Omega \subset \wh{G}$ stands for a fundamental domain for $\wh{G}/\Gamma^\perp$.
\end{lemma}
\begin{proof}
By Plancherel theorem on $L^2(G)$ and the unitarity of $\wh{D}$ we have
\begin{align*}
\mathcal{I}_{\Gamma,H}(f,\psi) & = \int_H \Big(\int_\Gamma |\langle \wh{f}, \wh{D}(h)\wh{T(\gamma)\psi}\rangle_{L^2(\wh{G})}|^2  d\mu_\Gamma(\gamma) \Big) d\varsigma(h)\\
& = \int_H \Big(\int_\Gamma |\langle \wh{D}(h)^{-1}\wh{f}, \wh{T(\gamma)\psi}\rangle_{L^2(\wh{G})}|^2  d\mu_\Gamma(\gamma) \Big) d\varsigma(h)
\end{align*}
where, by (\ref{eq:dualdilations}), $\wh{D}(h)^{-1} \wh{f}(\xi) = \delta(h)^{-\frac12} \wh{f}(\dalpha_h^{\,-1}(\xi))$, and
\[
\wh{T(\gamma)\psi}(\xi) = \int_G \psi(x - \gamma) \ol{\langle \xi, x \rangle} d\mu_G(x) = \ol{\langle \xi, \gamma \rangle} \wh{\psi}(\xi).
\]
Thus, by $i.$, Lemma \ref{lem:Weil}, we have
\begin{align*}
\langle \wh{D}(h)^{-1}\wh{f}, \wh{T(\gamma)\psi}&\rangle_{L^2(\wh{G})} = \delta(h)^{-\frac12} \int_{\wh{G}} \wh{f}(\dalpha_h^{\, -1}(\xi)) \ol{\wh{\psi}(\xi)} \langle \xi, \gamma \rangle d\nu_{\wh{G}}(\xi)\\
& = \delta(h)^{-\frac12} \int_{\Omega} \left(\sum_{\lambda \in \Gamma^\perp}\wh{f}(\dalpha_h^{\, -1}(\xi+\lambda)) \ol{\wh{\psi}(\xi+\lambda)}\right) \langle \xi, \gamma \rangle d\nu_{\wh{G}}(\xi) .
\end{align*}
Since $\mathcal{A}_{\Gamma,H}(\psi)$ is a Bessel system, then the left hand side in the above chain of identities, as a function of $\gamma$, belongs to $L^2(\Gamma)$ for $\varsigma$-a.e. $h \in H$, so that
\[
\xi \mapsto \sum_{\lambda \in \Gamma^\perp}\wh{f}(\dalpha_h^{\, -1}(\xi+\lambda)) \ol{\wh{\psi}(\xi+\lambda)}
\]
belongs to $L^2(\Omega)$, and the conclusion follows by $ii.$, Lemma \ref{lem:Weil}.
\end{proof}

\section{Boundedness of the Calder\'on's sum}\label{sec:boundedness}

In this section we will make use of the following standing assumptions:\vspace{1ex}\\
{\bf I)} The groups $G,\wh{G},\Gamma,\Gamma^\perp$ and their Haar measures are as in \S\ref{sec:preliminaries}.\\
{\bf II)} The group $\wh{G}$ has an invariant metric $\dg$ such that $(\wh{G},\dg,\nu_{\wh{G}})$ is weakly doubling\footnote{Note that this is sufficient to have Lebesgue's differentiation theorem, see \cite[Theorem 3.4.3]{Heinonen}. A different condition that is weaker than doubling and allows to obtain a Lebesgue's differentiation theorem is used in \cite[Theorem (44.18)]{HewittRossII} and \cite{BownikRoss}.}, i.e. that there exist $r_0 > 0$ and $C > 0$ such that
\begin{equation}\label{eq:Vitali}
\nu_{\wh{G}}(B(\id,2r)) \leq C \nu_{\wh{G}}(B(\id,r)) \quad \forall \ r < r_0.
\end{equation}
{\bf III)} The automorphisms map $\dalpha : H \to \textnormal{Aut}(\wh{G})$ is measurably bi-Lipschitz with respect to $\dg$ and, in order to make use of Fubini's theorem, the Borel measure $\varsigma$ on $H$ is $\sigma$-finite.

\subsection{The test space}
Following the original idea of \cite{ChuiShi}, for $\xi_0 \in \wh{G}$ and $\epsilon > 0$ let us consider the unit norm $f^{\xi_0}_\epsilon \in L^2(G)$ given by
\[
\wh{f^{\xi_0}_\epsilon} = \frac{1}{\sqrt{\nu_{\wh{G}} (B(\xi_0,\epsilon))}} \one{B(\xi_0,\epsilon)}.
\]
By Lemma \ref{lem:fourier} and Remark \ref{rem:measureballdilations}, whenever $\mathcal{A}_{\Gamma,H}(\psi)$ is a $\varsigma$-Bessel system for $\Hil \subset L^2(G)$, and if $f^{\xi_0}_\epsilon$ belongs to $\Hil$, we get
\begin{equation}\label{eq:ChuiShi}
\mathcal{I}_{\Gamma,H}(f^{\xi_0}_\epsilon,\psi) = \int_H \int_\Omega |\sum_{\lambda \in \Gamma^\perp} \one{\dalpha_hB(\xi_0,\epsilon)}(\xi + \lambda) \ol{\wh{\psi}(\xi + \lambda)}|^2  \frac{d\nu_{\wh{G}}(\xi) d\varsigma(h)}{\nu_{\wh{G}} (\dalpha_hB(\xi_0,\epsilon))} .
\end{equation}

If $\mathcal{A}_{\Gamma,H}(\psi)$ is a frame with constants $A$ and $B$, then (\ref{eq:ChuiShi}) is bounded from above and below by these same constants. We will show that it is possible to obtain from this representation formula the same bounds for the Calder\'on's sum in the limit for $\epsilon \to 0$.

Since we will be dealing with such a special class of test functions, whose $G$-Fourier transform is the characteristic function of a metric ball in $\wh{G}$, we can make precise the requirement $f^{\xi_0}_\epsilon \in \Hil$ with the following definition.
\begin{definition}\label{def:Gset}
Let $\Hil$ be a Hilbert subspace of $L^2(G)$ endowed with the induced norm.
Denoting by $\wh{\Hil} \subset L^2(\wh{G})$ the image of $\Hil$ under the $G$-Fourier transform, for $\epsilon_0 > 0$, define 
\[
\wh{\mathcal{G}}^{\,\Hil}_{\epsilon_0} = \{\xi \in \wh{G} \, : \, \one{B(\xi,\epsilon)} \in \wh{\Hil} \quad \forall \ \epsilon < \epsilon_0\}
\]
that is the set characterized by $f^{\xi_0}_\epsilon \in \Hil \ \iff \ \xi_0 \in \wh{\mathcal{G}}^{\,\Hil}_{\epsilon_0} \ \textnormal{and} \ \epsilon < \epsilon_0$.
\end{definition}

In order to avoid trivial pitfalls or too technical assumptions, we will always implicitly assume that such $\wh{\mathcal{G}}^{\,\Hil}_{\epsilon_0}$ has a positive $\nu_{\wh{G}}$ measure. In the next sections we will see that this is actually the case in relevant situations. As a preliminary toy example, let $B_r \subset \wh{G}$ be a ball of radius $r > \epsilon_0$. The Paley-Wiener space $\Hil_r = \{f \in L^2(G) \, : \, \textnormal{supp}\wh{f} \subset \ol{B_r}\}$ is such that $\wh{\mathcal{G}}^{\,\Hil_r}_{\epsilon_0} = B_{r - \epsilon_0}$.

A crucial observation, already present in \cite{ChuiShi}, is that, by restricting to a subset of $H$ which makes the automorphisms a uniformly bounded Lipschitz family, $\mathcal{I}_{\Gamma,H_M}(f^{\xi_0}_\epsilon,\psi)$ actually approximates the Calder\'on's sum for small $\epsilon$.

\begin{lemma}\label{lem:Lipschitz}
Let $\mathcal{A}_{\Gamma,H}(\psi)$ be an affine system as in Definition \ref{def:frames}. If $\mathcal{A}_{\Gamma,H}(\psi)$ is a $\varsigma$-Bessel system for $\Hil \subset L^2(G)$ and, for $M > 0$,  $H_M = \{h \in H \, : \, L(h) < M\}$, then there exists $\epsilon_0 > 0$ such that
\begin{equation}\label{eq:HM}
\mathcal{I}_{\Gamma,H_M}(f^{\xi_0}_\epsilon,\psi) = \frac{1}{\nu_{\wh{G}}(B(\xi_0,\epsilon))}\int_{B(\xi_0,\epsilon)} \int_{H_M} |\wh{\psi}(\dalpha_h(\xi))|^2 d\varsigma(h) d\nu_{\wh{G}}(\xi)
\end{equation}
for all $\xi_0 \in \wh{\mathcal{G}}^{\,\Hil}_{\epsilon_0}$ and all $\epsilon < \epsilon_0$.
\end{lemma}
\begin{proof}
Without loss of generality, let us choose a fundamental set $\Omega$ that contains a neighborhood of $\dalpha_h(\xi_0)$. Since, by (\ref{eq:balls}), we have that
\[
\dalpha_hB(\xi_0,\epsilon) \subset B(\dalpha_h(\xi_0),M\epsilon) \qquad \forall \ h \in H_M \, ,
\]
then, for $\epsilon$ sufficiently small, we get $\dalpha_hB(\xi_0,\epsilon) \subset \Omega$.
For such an $\epsilon$, and for $h \in H_M$, the sum in (\ref{eq:ChuiShi}) contains only one term. Indeed, if $\lambda \neq \id$ then $(\dalpha_hB(\xi_0,\epsilon) + \lambda) \cap \Omega = \emptyset$, so $\one{\dalpha_hB(\xi_0,\epsilon)}(\xi + \lambda) = 0$ for all $\xi \in \Omega$. Thus,
\begin{align*}
\mathcal{I}_{\Gamma,H_M}(f^{\xi_0}_\epsilon,\psi) & = \int_{H_M} \int_\Omega \one{\dalpha_hB(\xi_0,\epsilon)}(\xi) |\wh{\psi}(\xi)|^2 d\nu_{\wh{G}}(\xi) \frac{d\varsigma(h)}{\nu_{\wh{G}} (\dalpha_hB(\xi_0,\epsilon))}\\
& = \int_{H_M} \int_{\dalpha_hB(\xi_0,\epsilon)} |\wh{\psi}(\xi)|^2 d\nu_{\wh{G}}(\xi) \frac{d\varsigma(h)}{\nu_{\wh{G}} (\dalpha_hB(\xi_0,\epsilon))}\\
& = \int_{H_M} \int_{B(\xi_0,\epsilon)} |\wh{\psi}(\dalpha_h(\xi))|^2 d\nu_{\wh{G}}(\xi) \frac{d\varsigma(h)}{\nu_{\wh{G}}(B(\xi_0,\epsilon))} .
\end{align*}
Since $\mathcal{I}_{\Gamma,H_M}(f^{\xi_0}_\epsilon,\psi) \leq \mathcal{I}_{\Gamma,H}(f^{\xi_0}_\epsilon,\psi)$ is bounded by the $\varsigma$-Bessel hypothesis, the claim follows by Fubini's theorem.
\end{proof}

\subsection{Bessel bound}\label{sec:above}

The upper bound for the Calder\'on's sum can be obtained directly from (\ref{eq:HM}) without any additional assumption, as shown by the following proposition.

\begin{proposition}\label{prop:Bessel}
Let $\mathcal{A}_{\Gamma,H}(\psi)$ be a $\varsigma$-Bessel system for $\Hil \subset L^2(G)$ with constant $B$. Then $\mathcal{C}_\psi \leq B$ almost everywhere in $\wh{\mathcal{G}}^{\,\Hil}_{\epsilon_0}$.
\end{proposition}
\begin{proof}
By Lemma \ref{lem:Lipschitz} the function $\xi \mapsto \int_{H_M} |\wh{\psi}(\dalpha_h(\xi))|^2 d\varsigma(h)$ belongs to $L^1_\textnormal{loc}(\wh{\mathcal{G}}^{\,\Hil}_{\epsilon_0})$. Thus, by applying Lebesgue's differentiation theorem at the limit for $\epsilon \to 0$ and then taking the limit for $M \to \infty$ in (\ref{eq:HM}), we get
\[
\lim_{M \to \infty} \lim_{\epsilon \to 0} \mathcal{I}_{\Gamma,H_M}(f^{\xi_0}_\epsilon,\psi) = \mathcal{C}_\psi(\xi_0) \quad \textnormal{for } \nu_{\wh{G}}\textnormal{-a.e. } \xi_0 \in \wh{\mathcal{G}}^{\,\Hil}_{\epsilon_0}.
\]
The result follows because, for all $\epsilon > 0$ and all $M > 0$. we have 
\[
\mathcal{I}_{\Gamma,H_M}(f^{\xi_0}_\epsilon,\psi) \leq \mathcal{I}_{\Gamma,H}(f^{\xi_0}_\epsilon,\psi) \leq B. \qedhere
\]
\end{proof}

\subsection{Main Theorem}\label{sec:below}

In order to obtain the bound from below of Calder\'on's sum for frames, we will need a technical assumption on the estimate of the number of lattice points inside a metric ball deformed by the family of automorphisms. A similar estimate is the central object of study in the recent paper \cite{BownikLemvig} dealing with affine wavelets in $\R^n$. Its role for the study of affine systems was also noted in \cite[Lemma 5.11]{HLW} and, for LCA groups, in \cite[Lemma 4.11]{KutyniokLabate}, but it could be found in disguise even in \cite[page 271]{ChuiShi}.


\begin{definition}\label{def:PropertyX}
Let $\Gamma$ be a cocompact subgroup of $G$. We say that a measurably bi-Lipschitz family $\{\dalpha_h\}_{h \in H}$ of $\wh{G}$-automorphisms has \emph{Property $\textnormal{X}$} for $\Gamma^\perp$ if there exist $r > 0$, $M > 0$, and a constant $C > 0$ such that
\begin{equation}\label{eq:PropertyX}
\sharp \Big(\Gamma^\perp \cap \dalpha_h B(\id,r)\Big) \leq 1 + C \delta(h)
\end{equation}
for all $h \in H_{M}^c = \{h \in H \, : \, L(h) > M\}$.
\end{definition}


\begin{remark}
Observe that the estimate (\ref{eq:PropertyX}) takes into account that at least the lattice point $\id \in \Gamma^\perp \cap \dalpha_h B(\id,\epsilon)$ has to be counted. Moreover, if (\ref{eq:PropertyX}) holds for a given constant $C$, then it holds with the same constant $C$ for all $r' < r$ and all $M' > M$.
\end{remark}

In the following sections we will deduce Property X for relevant systems from our counting estimate of Lemma \ref{lem:counting}. However, a simple paradigmatic example of a setting where it does not hold may be considered now.

\begin{example}\label{ex:BAD}
Let $\wh{G} = \R^2$, let $\Gamma^\perp = \Z^2$, let $H = \Z$ and let $\alpha_j = \binom{2 \ \, 0}{0 \ \frac12}^j$. Then $\delta(j) = \det(\alpha_j) = 1$ and $H_M^c = \{|j| > \log_2 M\}$. That (\ref{eq:PropertyX}) does not hold can be seen because, for any fixed $\epsilon$, for large values of $j$ the deformed balls are allowed to contain an arbitrarily large number of lattice points.
\end{example}

The main step in order to obtain the lower bound consists of obtaining an approximation to the full Calder\'on's sum, up to a remainder, from the representation formula (\ref{eq:ChuiShi}).

\begin{proposition}\label{prop:remainder}
Let $\mathcal{A}_{\Gamma,H}(\psi)$ be a $\varsigma$-frame of $\Hil \subset L^2(G)$  with constants $0 < A \leq B < \infty$, and let $\{\dalpha_h\}_{h \in H}$ have Property X. Then there exist an $\epsilon_0 > 0$ and an $M_0 > 0$ such that
\[
A \leq \frac{1}{\nu_{\wh{G}}(B(\xi_0,\epsilon))}\int_{B(\xi_0,\epsilon)} \int_{H} |\wh{\psi}(\dalpha_h(\xi))|^2 d\varsigma(h) d\nu_{\wh{G}}(\xi) + R_M^\psi(\epsilon,\xi_0)
\]
for all $\xi_0 \in \wh{\mathcal{G}}^{\,\Hil}_{\epsilon_0}$, all $\epsilon < \epsilon_0$, and all $M > M_0$, with
\begin{equation}\label{eq:remainder}
R_M^\psi(\epsilon,\xi_0) = \frac{C}{\nu_{\wh{G}} (B(\xi_0,\epsilon))}\int_{H_M^c} \int_{B(\xi_0,\epsilon)} |\wh{D}(h)\wh{\psi}(\xi)|^2 d\nu_{\wh{G}}(\xi) d\varsigma(h)
\end{equation}
where $C$ is given by (\ref{eq:PropertyX}) and $\wh{D}$ was defined in (\ref{eq:dualdilations}).
\end{proposition}
\begin{proof}
Since $H = H_M \sqcup H_M^c$, then
\begin{equation}\label{eq:simplesum}
\mathcal{I}_{\Gamma,H}(f^{\xi_0}_\epsilon,\psi) = \mathcal{I}_{\Gamma,H_M}(f^{\xi_0}_\epsilon,\psi) + \mathcal{I}_{\Gamma,H_M^c}(f^{\xi_0}_\epsilon,\psi).
\end{equation}
The first term can be treated by Lemma \ref{lem:Lipschitz}, so for some $\epsilon_0 > 0$ we have
\[
\mathcal{I}_{\Gamma,H_M}(f^{\xi_0}_\epsilon,\psi) = \frac{1}{\nu_{\wh{G}}(B(\xi_0,\epsilon))}\int_{B(\xi_0,\epsilon)} \int_{H_M} |\wh{\psi}(\dalpha_h(\xi))|^2 d\varsigma(h) d\nu_{\wh{G}}(\xi)
\]
for all $\epsilon < \epsilon_0$. Let us then focus on the second term. By Cauchy-Schwarz inequality we can write, for the inner summation in (\ref{eq:ChuiShi}),
\[
|\sum_{\lambda \in \Gamma^\perp} \one{\dalpha_hB(\xi_0,\epsilon)}(\xi + \lambda) \ol{\wh{\psi}(\xi + \lambda)}|^2 \leq S_\epsilon(h,\xi) \sum_{\lambda \in \Gamma^\perp} \one{\dalpha_hB(\xi_0,\epsilon)}(\xi + \lambda) |\wh{\psi}(\xi + \lambda)|^2
\]
where
\[
S_\epsilon(h,\xi) = \sum_{\lambda \in \Gamma^\perp} \one{\dalpha_hB(\xi_0,\epsilon)}(\xi + \lambda) .
\]
Note that this is actually a finite sum. Moreover, for all $\lambda \in \Gamma^\perp$ we have $S_\epsilon(h,\xi) = S_\epsilon(h,\xi+\lambda)$, so by $i.$, Lemma \ref{lem:Weil} we get
\begin{align*}
& \mathcal{I}_{\Gamma,H_M^c}(f^{\xi_0}_\epsilon,\psi) \leq
\int_{H_M^c}\int_\Omega S_\epsilon(h,\xi) \sum_{\lambda \in \Gamma^\perp} \one{\dalpha_hB(\xi_0,\epsilon)}(\xi + \lambda) |\wh{\psi}(\xi + \lambda)|^2 \frac{d\nu_{\wh{G}}(\xi)d\varsigma(h)}{\nu_{\wh{G}} (\dalpha_hB(\xi_0,\epsilon))}\\
& = \int_{H_M^c}\int_\Omega \sum_{\lambda \in \Gamma^\perp} S_\epsilon(h,\xi+\lambda)\one{\dalpha_hB(\xi_0,\epsilon)}(\xi + \lambda) |\wh{\psi}(\xi + \lambda)|^2  \frac{d\nu_{\wh{G}}(\xi) d\varsigma(h)}{\nu_{\wh{G}} (\dalpha_hB(\xi_0,\epsilon))}\\
& = \int_{H_M^c}\int_{\wh{G}} S_\epsilon(h,\xi)\one{\dalpha_hB(\xi_0,\epsilon)}(\xi) |\wh{\psi}(\xi)|^2  \frac{d\nu_{\wh{G}}(\xi) d\varsigma(h)}{\nu_{\wh{G}} (\dalpha_hB(\xi_0,\epsilon))}\\
& = \int_{H_M^c}\int_{\dalpha_hB(\xi_0,\epsilon)} S_\epsilon(h,\xi) |\wh{\psi}(\xi)|^2 \frac{d\nu_{\wh{G}}(\xi) d\varsigma(h)}{\nu_{\wh{G}} (\dalpha_hB(\xi_0,\epsilon))} .
\end{align*}
By the group invariance of the metric $\dg$, when $\xi \in \dalpha_hB(\xi_0,\epsilon)$ we have
\begin{align}\label{eq:estimateV}
S_\epsilon(h,\xi) & \leq \sharp \Big\{\lambda \in \Gamma^\perp \, : \, \dalpha_hB(\xi_0,\epsilon) \cap (\dalpha_hB(\xi_0,\epsilon) + \lambda) \neq \emptyset\Big\}\nonumber\\
& = \sharp \Big\{\lambda \in \Gamma^\perp \, : \, \dalpha_hB(\xi_0,\epsilon) \cap \dalpha_hB(\xi_0 + \dalpha_h^{\, -1}(\lambda),\epsilon) \neq \emptyset\Big\}\nonumber\\
& = \sharp \Big\{\lambda \in \Gamma^\perp \, : \, B(\xi_0,\epsilon) \cap B(\xi_0 + \dalpha_h^{\, -1}(\lambda),\epsilon) \neq \emptyset\Big\}\nonumber\\
& \leq \sharp \Big\{\lambda \in \Gamma^\perp \, : \, \dg(\dalpha_h^{\, -1}(\lambda),\id) < 2\epsilon \Big\} = \sharp \Big(\dalpha_h^{\, -1} \Gamma^\perp \,\cap\, B(\id,2\epsilon)\Big)\nonumber\\
& = \sharp \Big(\dalpha_h^{\, -1} \big(\Gamma^\perp \,\cap\, \dalpha_hB(\id,2\epsilon)\big)\Big) = \sharp \Big(\Gamma^\perp \,\cap\, \dalpha_hB(\id,2\epsilon)\Big)\nonumber\\
& \leq 1 + C \delta(h)
\end{align}
where the last inequality is Property X.

From (\ref{eq:simplesum}) and (\ref{eq:estimateV}) and arguing as in Lemma \ref{lem:Lipschitz}, we have then obtained
\[
\mathcal{I}_{\Gamma,H}(f^{\xi_0}_\epsilon,\psi) \leq \frac{1}{\nu_{\wh{G}}(B(\xi_0,\epsilon))}\int_{B(\xi_0,\epsilon)} \int_{H} |\wh{\psi}(\dalpha_h(\xi))|^2 d\varsigma(h) d\nu_{\wh{G}}(\xi) + R_M^\psi(\epsilon,\xi_0)
\]
where the remainder term reads
\[
R_M^\psi(\epsilon,\xi_0) = C \int_{H_M^c} \int_{\dalpha_hB(\xi_0,\epsilon)} |\wh{\psi}(\xi)|^2 \frac{\delta(h)}{\nu_{\wh{G}} (\dalpha_hB(\xi_0,\epsilon))} d\nu_{\wh{G}}(\xi) d\varsigma(h)
\]
which coincides with (\ref{eq:remainder}) by Remark \ref{rem:measureballdilations}.
\end{proof}

We are now ready to prove our main theorem, which provides sharp bounds for the Calder\'on's sum as the frame bounds. For this, we require an hypothesis of local integrability for a term which, whenever $\delta \neq 1$, differs from the Calder\'on's sum by such a factor. The role of this hypothesis has been discussed in a recent paper concerning generalized shift-invariant systems on the real line \cite{ChristensenHasannasabLemvig}, where the authors prove that it is strictly related to the so-called LIC condition introduced in \cite{HLW} and $\alpha$-LIC condition introduced in \cite{JakobsenLemvigTranslations}. In the present setting, it is  the minimal assumption needed to make use of Fubini's theorem and prove that the remainder vanishes in the limit. In the next section, we will see that this local integrability condition holds for rather general classes of automorphisms.
\begin{theorem}\label{theo:main}
Let $\mathcal{A}_{\Gamma,H}(\psi)$ be a $\varsigma$-frame of $\Hil \subset L^2(G)$ with constants $0 < A \leq B < \infty$, and suppose that $\{\dalpha_h\}_{h \in H}$ is a measurably bi-Lipschitz family of automorphisms, with upper Lipschitz constants $\{L(h)\}_{h \in H}$, satisfying Property X. If there exist $M>0$ and $\epsilon_0 > 0$ such that
\begin{equation}\label{eq:localintegrability}
\Psi_M(\xi) = \int_{H_M^c} |\wh{D}(h)\wh{\psi}(\xi)|^2 d\varsigma(h) \ \in L^1_\textnormal{loc}(\wh{\mathcal{G}}^{\,\Hil}_{\epsilon_0}\smallsetminus \mathcal{O})
\end{equation}
for some zero-measure set $\mathcal{O}$, where $H_M^c = \{h \in H \, : \, L(h) > M\}$ and $\wh{\mathcal{G}}^{\,\Hil}_{\epsilon_0}$ is as in Definition \ref{def:Gset}, then
\[
A \leq \mathcal{C}_\psi(\xi) \leq B \quad \textnormal{for} \ \nu_{\wh{G}}-\textnormal{a.e.} \ \xi \in \wh{\mathcal{G}}^{\,\Hil}_{\epsilon_0}.
\]
\end{theorem}

\begin{proof}
The upper bound is provided by Proposition \ref{prop:Bessel}. For the lower bound, if (\ref{eq:localintegrability}) holds, then by Fubini's theorem the remainder term (\ref{eq:remainder}) reads
\[
R_M^\psi(\epsilon,\xi_0) = \frac{C}{\nu_{\wh{G}} (B(\xi_0,\epsilon))}\int_{B(\xi_0,\epsilon)} \Psi_M(\xi) d\nu_{\wh{G}}(\xi) .
\]
Since, by Proposition \ref{prop:remainder}, we have
\[
A \leq \frac{1}{\nu_{\wh{G}}(B(\xi_0,\epsilon))}\int_{B(\xi_0,\epsilon)} \int_{H} |\wh{\psi}(\dalpha_h(\xi))|^2 d\varsigma(h) d\nu_{\wh{G}}(\xi) + R_M^\psi(\epsilon,\xi_0)
\]
and we can apply Lebesgue's differentiation theorem to both terms, by taking the $\lim$ for $\epsilon \to 0$ we get
\[
A \leq \mathcal{C}_\psi(\xi_0) + C \,\Psi_{M}(\xi_0) \quad \textnormal{a.e.} \ \xi_0 \in \wh{\mathcal{G}}^{\,\Hil}_{\epsilon_0}.
\]
Since the only term depending on $M$ is $\Psi_{M}$, the proof is concluded by showing that this quantity vanishes when $M$ tends to infinity. Note that if condition (\ref{eq:localintegrability}) holds for a given $M$, it holds for all larger values of $M$.

In order to see that, let $M \in \mathbb{N}$, so that can write $H_M^c$ as the disjoint union
\[
H_M^c = \bigsqcup_{n=M}^\infty \{h \in H \, : \, n < L(h) \leq n+1\} = \bigsqcup_{n=M}^\infty \Sigma_n .
\]
Denoting by $\phi_n(\xi_0) = \displaystyle\int_{\Sigma_n} |\wh{D}(h)\wh{\psi}(\xi_0)|^2 d\varsigma(h)$, we then get
\[
\Psi_M(\xi_0) = \sum_{n=M}^\infty \phi_n(\xi_0) .
\]
By (\ref{eq:localintegrability}), for a.e. $\xi_0 \in \wh{\mathcal{G}}^{\,\Hil}_{\epsilon_0}$ we have that $\Psi_M(\xi_0)$ is finite. So, in particular, each term $\phi_n(\xi_0)$ is finite, and $\lim_{n \to \infty} \phi_n(\xi_0) = 0$ for a.e. $\xi_0 \in \wh{\mathcal{G}}^{\,\Hil}_{\epsilon_0}$. Thus
\[
\lim_{M \to \infty} \Psi_{M}(\xi_0) = 0
\]
concluding the proof.
\end{proof}

\section{Expanding automorphisms}

In this section we will consider $\Hil$ to be the whole $L^2(G)$, so that $\wh{\mathcal{G}}^{\,\Hil}_{\epsilon_0} = \wh{G}$ for any $\epsilon_0 > 0$. We will make use of a weak notion of expanding automorphisms, which we introduce in terms of the Lipschitz constants defined in (\ref{eq:metric}).
\begin{definition}\label{def:expanding}
A measurably bi-Lipschitz family of $\wh{G}$-auto\-morphisms $\{\dalpha_h\}_{h \in H}$, with bi-Lipschitz constants $\{\ell(h)\}_{h \in H}$ and $\{L(h)\}_{h \in H}$, is called \emph{expanding} if
there exist $M_0, N_0 > 0$ such that, for all $M > M_0$
\[
L(h) > M \Rightarrow \ell(h) > N_0 .
\]
We say that it is \emph{uniformly expanding} if there exist an $M > 0$ and a monotone increasing function $f : \R^+ \to \R^+$ such that
\[
L(h) > M \Rightarrow L(h) \leq f(\ell(h)) .
\] 
In particular, if $\{\dalpha_h\}_{h \in H}$ is uniformly expanding, then it is expanding.
\end{definition}
This notion of expanding imposes that $\ell(h)$ can not be arbitrarily small when $L(h)$ is arbitrarily large. This means that any automorphism of the family can not put near two points in the space while moving away two other points at arbitrary distances. It thus rules out cases such as that of Example \ref{ex:BAD}. This general formulation is compatible with the more classical ones considered in \cite{HLW}, \cite{GuoLabate} and \cite{BownikLemvig}, in the sense that the expanding matrices of these works are expanding for the current definition, see Proposition \ref{prop:expandingsubspace}.

\begin{remark}
This notion of expanding does not require the deformation ratio $\displaystyle{\frac{L(h)}{\ell(h)}}$ to be  bounded on $H$, as this would exclude relevant cases such as anisotropic dilations. This ratio, which defines the quasiconformality (see e.g. \cite[\S 14.2]{Heinonen}) of $\dalpha_h$, is indeed allowed to be arbitrarily large with $h$ in $H$.
\end{remark}

\subsection{Property X and Calder\'on's bounds}
We show here that the introduced notion of expanding automorphisms allows us to obtain Property X from the counting Lemma \ref{lem:counting}.

\begin{theorem}\label{theo:expandingX}
Let $\{\dalpha_h\}_{h \in H}$ be an expanding family of $\wh{G}$-auto\-morphisms. Then, for all $\Gamma$ cocompact subgroup of $G$, it has Property X for $\Gamma^\perp$.
\end{theorem}
\begin{proof}
Let $\Gamma$ be a cocompact subgroup of $G$.
We introduce the following shorthand notation derived from (\ref{eq:neighborhood}): if $\Omega$ is a fundamental domain for $\Gamma^\perp$ in $\wh{G}$, and $r > 0$, let us call
\begin{equation}\label{eq:littleomega}
\omega(r) = \nu_{\wh{G}} \left(\Omega^r_{\mathbf{1}}\right) = \nu_{\wh{G}} \Big(\Omega \cap \bigcup_{\lambda \in \Gamma^\perp} B(\lambda,r)\Big) .
\end{equation}
By (\ref{eq:balls}) and Lemma \ref{lem:counting}, and recalling Remark \ref{rem:measureballdilations}, we have that for all $N > 0$
\[
\sharp \Big(\Gamma^\perp \cap \dalpha_h B(\id,r)\Big) \leq \frac{\nu_{\wh{G}}(B(\id,2r))}{\omega(Nr)}\delta(h) \quad \forall \ r > 0 \, , \ \forall \ h \in K_N.
\]
where $K_N = \{h \in H \, : \, \ell(h) > N \}$. 
Let us observe now that we can choose $\epsilon$ and $N$ so that
\[
\omega(N\epsilon) = \nu_{\wh{G}}(B(\id,N\epsilon)).
\]
Indeed, without loss of generality assume that $\Omega$ contains a neighborhood of $\id \in \Gamma^\perp$. Then for $\epsilon$ sufficiently small we get $B(\id,\epsilon) \subset \Omega$, and no other ball $B(\lambda,\epsilon)$ in (\ref{eq:littleomega}), for $\lambda \neq \id$, intersects $\Omega$. Therefore, for all $h \in K_N$
\[
\sharp \Big(\Gamma^\perp \cap \dalpha_h B(\id,\epsilon)\Big) \leq \frac{\nu_{\wh{G}}(B(\id,2\epsilon))}{\nu_{\wh{G}}(B(\id,N\epsilon))}\delta(h) \leq 1 + \frac{\nu_{\wh{G}}(B(\id,2\epsilon))}{\nu_{\wh{G}}(B(\id,N\epsilon))} \delta(h).
\]
Now, by the weak doubling property (\ref{eq:Vitali}), if we assume without loss of generality that $N = 2^{-q}$, we get
$$
\frac{\nu_{\wh{G}}(B(\id,2\epsilon))}{\nu_{\wh{G}}(B(\id,N\epsilon))} \leq C^{q+1} .
$$
Thus, we have proved that for any $\epsilon$ sufficiently small, and for any $N > 0$, there exists $C$ such that
$$
\sharp \Big(\Gamma^\perp \cap \dalpha_h B(\id,\epsilon)\Big) \leq 1 + C \delta(h)
$$
for all $h \in K_N$. The proof now follows as a consequence of the expanding property, which implies that there exists an $M > 0$ such that  $H_M^c = \{ h \in H \, : \, L(h) > M\} \subset K_{N}$.
\end{proof}

As a consequence of Theorems \ref{theo:main} and \ref{theo:expandingX} we get then the following Calder\'on's inequalities.

\begin{corollary}\label{cor:expandingCalderon}
Let $\mathcal{A}_{\Gamma,H}(\psi)$ be a $\varsigma$-frame of $L^2(G)$  with constants $0 < A \leq B < \infty$, and let $\{\dalpha_h\}_{h \in H}$ be an expanding family of bi-Lipschitz automorphisms, with upper Lipschitz constants $\{L(h)\}_{h \in H}$. If
\begin{equation}\label{eq:localintegrabilityII}
\Psi_M(\xi) = \int_{H_M^c} |\wh{D}(h)\wh{\psi}(\xi)|^2 d\varsigma(h) \ \in L^1_\textnormal{loc}(\wh{G}\smallsetminus \mathcal{O})
\end{equation}
for some zero-measure set $\mathcal{O}$ and some $M > 0$, where $H_M^c = \{ h \in H \, : \, L(h) > M\}$, then
\[
A \leq \mathcal{C}_\psi(\xi) \leq B \quad \textnormal{for} \ \nu_{\wh{G}}-\textnormal{a.e.} \ \xi \in \wh{G}.
\]
\end{corollary}

\subsection{Sufficient conditions for local integrability}
We provide here a sufficient condition for the local integrability (\ref{eq:localintegrabilityII}) in the case of a uniformly expanding family of automorphisms.
\begin{theorem}\label{theo:sufficient}
Let $\{\dalpha_h\}_{h \in H}$ be uniformly expanding. If $\,\exists\, M>0$ such that
\begin{equation}\label{eq:sufficient}
u_c(t) = \varsigma\left(\left\{h \in H \, : \, t \leq L(h) \leq f(c t)\right\}\right) \in L^\infty\big([M,+\infty)\big)
\end{equation}
for all $c > 1$, then (\ref{eq:localintegrabilityII}) is satisfied by all $\psi \in L^2(G)$, with $\mathcal{O}=\{\id\}$.
\end{theorem}
\begin{proof}
We prove local integrability by showing that
\[
E_M(\epsilon,\xi_0) = \int_{H_M^c} \int_{\dalpha_hB(\xi_0,\epsilon)} |\wh{\psi}(\xi)|^2 d\nu_{\wh{G}}(\xi) d\varsigma(h) < \infty \quad \forall \xi_0 \neq \id .
\]
Let $\xi \in \dalpha_hB(\xi_0,\epsilon)$, i.e. $\dg(\dalpha_h^{\, -1}(\xi),\xi_0)<\epsilon$.
By (\ref{eq:metric}) and (\ref{eq:balls}) we have
\begin{align*}
\dg(\xi,\id) & \leq \dg(\dalpha_h(\xi_0),\id) + \dg(\xi,\dalpha_h(\xi_0)) \leq L(h) (\dg(\xi_0,\id) + \dg(\dalpha_h^{\, -1}(\xi),\xi_0))\\
& \leq L(h) (\dg(\xi_0,\id) + \epsilon)\\
\dg(\xi,\id) & \geq \ell(h) \dg(\dalpha_h^{\, -1}(\xi),\id) \geq \ell(h) (\dg(\xi_0,\id) - \dg(\dalpha_h^{\, -1}(\xi),\xi_0))\\
& \geq \ell(h)(\dg(\xi_0,\id) - \epsilon) \geq f^{-1}(L(h))(\dg(\xi_0,\id) - \epsilon).
\end{align*}
Denoting by $\mathfrak{C}(r,R) = B(\id,R)\smallsetminus B(\id,r)$, we can then write
\[
E_M(\epsilon,\xi_0) \leq \int_{H_M^c} \int_{\mathfrak{C}(r_h,R_h)} |\wh{\psi}(\xi)|^2 d\nu_{\wh{G}}(\xi) d\varsigma(h)
\]
where $r_h = f^{-1}(L(h))(d(\xi_0,\id) - \epsilon)$ and $R_h = L(h)(d(\xi_0,\id)+\epsilon)$.

The domain of integration can now be split as in Figure \ref{fig:fubini}, so that
\[
E_M(\epsilon,\xi_0) \leq \int_{A \sqcup A'} |\wh{\psi}(\xi)|^2 d\nu_{\wh{G}}(\xi) d\varsigma(h) + \int_B |\wh{\psi}(\xi)|^2 d\nu_{\wh{G}}(\xi) d\varsigma(h) = E^{(1)} + E^{(2)}.
\]
\begin{figure}[h!]
\centering
\begin{tikzpicture}[scale=2.7]
 \draw[->] (-0.2,0) -- (4.2,0) coordinate (x axis);
 \draw[->] (0,-0.2) -- (0,1.52) coordinate (y axis);
 
 \node at (4,-0.16) {$L(h)$};
 \node at (-0.35,1.4) {$\dg(\xi,\id)$};
 \node at (0.34,1.25) {$\textstyle H_M$};
 \node at (0.66,1.24) {$\textstyle H_M^c$};

 \draw[black, thick, domain=0:4.2] plot (\x, {0.3*(\x + 0.2*sin(3*\x r))});
 \draw[black, thick, domain=0:1.6] plot (\x, {0.8*\x});
 
 \draw[gray, dashed] (0.5,0) -- (0.5,1.48);
 \draw[gray, dashed] (0,0.4) -- (1.7,0.4);
 \draw[gray, dashed] (1.53,0) -- (1.53,0.5);
 \draw[gray, dashed] (0,0.21) -- (1.7,0.21);
 \draw[lightgray, line width = 2, line cap=round] (1.529,1.2) -- (4.028,1.2);

 \draw (0.5,1.5pt) -- (0.5,-1.5pt);
 \node at (0.5,-0.142) {$\scriptstyle M$};
 \draw (1.53,1.5pt) -- (1.53,-1.5pt);
 \node at (1.5,-0.16) {$\scriptstyle f(\frac{\Delta + \epsilon}{\Delta - \epsilon}M)$};
 \draw (0.05,0.4) -- (-0.05,0.4);
 \node at (-0.32,0.4) {$\scriptstyle (\Delta + \epsilon)M$};
 \draw (0.05,0.21) -- (-0.05,0.21);
 \node at (-0.45,0.21) {$\scriptstyle (\Delta - \epsilon)f^{-1}(M)$};
 
 \node at (2.8,1.35) {$\frac{\dg(\xi,\id)}{\Delta + \epsilon} \leq L(h) \leq f\big(\frac{\dg(\xi,\id)}{\Delta - \epsilon}\big)$};
 \node at (1.52,0.8) {$\scriptscriptstyle\dg(\xi,\id)=(\Delta + \epsilon) L(h)$};
 \node at (3.22,0.8) {$\scriptscriptstyle\dg(\xi,\id)=(\Delta - \epsilon) f^{-1}(L(h))$};

 \node at (0.59,0.322) {$A$};
 \node at (1.42,0.29) {$A'$};
 \node at (1.2,0.55) {$B$};

\end{tikzpicture}
\caption{Domain of integration for $E_M(\epsilon,\xi_0)$, with $\Delta = \dg(\xi_0,\id)$.}\label{fig:fubini}
\end{figure}
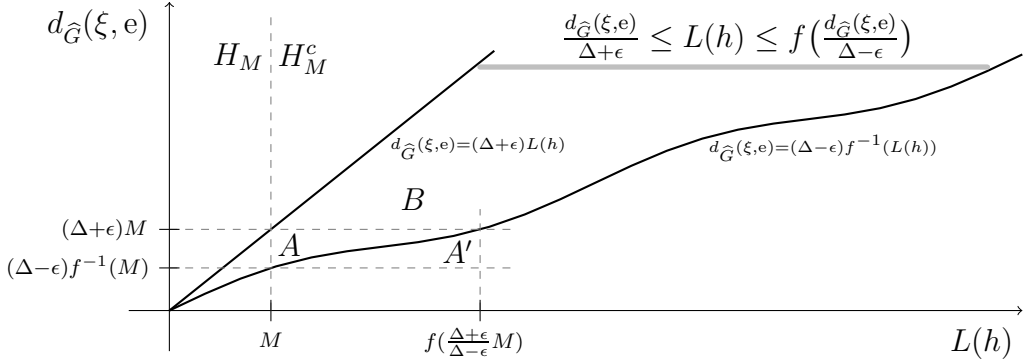

For the first term, using Fubini's theorem, we exchange the integrations in $\wh{G}$ and $H$ as follows: denoting by $F(t_1,t_2) = \varsigma\left(\left\{h \in H_M^c \, : \, t_1 \leq L(h) \leq t_2\right\}\right)$
\begin{align*}
E^{(1)} & = F\left(M,f\bigg(\frac{\Delta+\epsilon}{\Delta-\epsilon}M\bigg)\right)\int_{\mathfrak{C}\big((\Delta - \epsilon)f^{-1}(M),(\Delta + \epsilon)M\big)} |\wh{\psi}(\xi)|^2 d\nu_{\wh{G}}(\xi)\\
& = u_c(M) \int_{\mathfrak{C}\big((\Delta - \epsilon)f^{-1}(M),(\Delta + \epsilon)M\big)} |\wh{\psi}(\xi)|^2 d\nu_{\wh{G}}(\xi)
\end{align*}
where we have used the shorthand notation $\Delta = \dg(\xi_0,\id)$, and $c = \displaystyle \frac{\Delta+\epsilon}{\Delta-\epsilon}$.

For the second term we can proceed analogously, obtaining
\begin{align*}
E^{(2)} & = \int_{\wh{G} \smallsetminus B(\id,(\Delta+\epsilon)M)} |\wh{\psi}(\xi)|^2 F\left(\frac{\dg(\xi,\id)}{\Delta + \epsilon},f\bigg(\frac{\dg(\xi,\id)}{\Delta - \epsilon}\bigg)\right) d\nu_{\wh{G}}(\xi)\\
& = \int_{\wh{G} \smallsetminus B(\id,(\Delta+\epsilon)M)} |\wh{\psi}(\xi)|^2 \, u_c\left(\frac{\dg(\xi,\id)}{\Delta + \epsilon}\right) d\nu_{\wh{G}}(\xi)
\end{align*}
again with $c = \displaystyle \frac{\Delta+\epsilon}{\Delta-\epsilon}$.
Assuming (\ref{eq:sufficient}), we then have
\[
E_M(\epsilon,\xi_0) \leq \Big(\ \textnormal{ess}\!\!\!\!\!\!\!\!\sup_{[M,+\infty)\ \quad} \!\!\!\!\!\!\!u_c \Big)\int_{\wh{G} \smallsetminus B(\id,M(\Delta-\epsilon))} |\wh{\psi}(\xi)|^2 d\nu_{\wh{G}}(\xi)
\]
which is finite.
\end{proof}

\subsection{Examples}\label{sec:examples}

In order to have a concrete picture of the presented results we briefly discuss here some relatively simple examples of wavelet-type frames.

\subsubsection{Semi-continuous wavelets on $L^2(\R)$}\ \\
Let $G = \R$, consider as discrete subgroup $\Gamma = \Z$, and as a set of automorphisms the full dilations group $H = \R^+$, with $\dalpha_a\xi = a\xi$ and a measure that is absolutely continuous with respect to the Lebesgue one, i.e. $d\varsigma(a) = m(a)da$ for some positive $m \in L^1(\R^+)$.
In this case, $\ell(a) = L(a) = a$, and $\dalpha$ is uniformly expanding with $f(x) = x$.
Condition (\ref{eq:localintegrability}) reads
\[
\Psi_M(\xi) = \int_M^{+\infty} |\sqrt{a} \wh{\psi}(a\xi)|^2 \, m(a)da = \frac{1}{\xi^2} \int_{M\xi}^{+\infty} |\wh{\psi}(\eta)|^2 \, \eta \, m\big(\frac{\eta}{\xi}\big)d\eta
\]
which is convergent for all $\psi \in L^2(G)$ and $\xi \neq 0$ only if $m(a) \stackrel{a \to \infty}{\lesssim} O(\frac{1}{a})$.
This is the same result one gets from condition (\ref{eq:sufficient}), since
\[
u_c(t) = \int_t^{ct} m(a)da .
\]
In this case, the Calder\'on's integral reads
\[
\mathcal{C}_\psi(\xi) = \int_M^{+\infty} |\wh{\psi}(a\xi)|^2 \, m(a)da.
\]

\subsubsection{Discrete wavelets on $L^2(\R^n)$ with anisotropic dilations}\ \\
Let $G = \R^n$, and let $\Gamma = \Z^n$. As set of automorphisms let us take $H = \Z$, with $\dalpha_j\xi = A^j\xi$ for a matrix $A \in GL_n(\R)$ (see also Appendix \ref{sec:more} for the case of so-called expanding on a subspace matrices).

By \cite[Lemma 5.1]{HLW}, there exist $0 < \ul{\lambda} < \ol{\lambda}$ and $C > 0$ such that the bi-Lipschitz constants are
\begin{align*}
\ell(j) = \frac{1}{C}\ul{\lambda}^j, L(j) = C \ol{\lambda}^j & \ , \quad j \geq 0\\
\ell(j) = \frac{1}{C}\ol{\lambda}^j, L(j) = C \ul{\lambda}^j & \ , \quad j \leq 0 .
\end{align*}
So if either $\ul{\lambda} > 1$ or $\ol{\lambda} < 1$, then $\dalpha$ is uniformly expanding, with
\[
f(x) = \left\{
\begin{array}{cc}
C (Cx)^{\log_{\ul{\lambda}} \ol{\lambda}} & x > 1\\
C (Cx)^{\log_{\ol{\lambda}} \ul{\lambda}} & x \leq 1
\end{array}
\right..
\]

Let $\ul{\lambda} > 1$, and set $M > 1$. The quantity $u_c$ in (\ref{eq:sufficient}) reads (we let the constant $c$ change freely)
\begin{align*}
u_c(t) & = \varsigma(\{j \in \Z \, : \, t \leq C\ol{\lambda}^j \leq C (C c t)^{\log_{\ul{\lambda}}\ol{\lambda}}\})\\
& = \varsigma(\{j \in \Z \, : \, \log_{\ol{\lambda}} \frac{t}{C} \leq j \leq (\log_{\ul{\lambda}}\ol{\lambda})(\log_{\ol{\lambda}} Cct)\})
\end{align*}
that is the measure of a finite set. Thus, if we choose $\varsigma = m(j) \cdot$ counting, condition (\ref{eq:sufficient}) reads simply
$\limsup_{j \to \infty} m(j) < \infty$

For comparison, let us try to check directly condition (\ref{eq:localintegrability}).
The dilation operator for $\psi \in L^2(\R^n)$ reads
\[
\wh{D}(j)\wh{\psi}(\xi) = |\det A|^\frac{j}{2}\psi(A^j\xi)
\]
with $\delta(j) = |\det A|^j$ so, for $N(M) = \lceil \log_{\ol{\lambda}} \frac{M}{C}\rceil$, we have
\[
\Psi_M(\xi) = \sum_{j = N(M)}^{+\infty} |\det A|^j|\wh{\psi}(A^j\xi)|^2 \, m(j)
\]
and, for any $K \subset \R^n$ compact,
\[
\int_{K} \Psi_M(\xi) d\xi = \int_{K} \sum_{j = N(M)}^{+\infty} \delta(j)|\wh{\psi}(A^j\xi)|^2 \, m(j) d\xi = \!\!\sum_{j = N(M)}^{+\infty} m(j) \int_{A^jK} |\wh{\psi}(\xi)|^2 d\xi .
\]
Let $d = \min_{\xi \in K} |\xi|$ and $D = \max_{\xi \in K}|\xi|$, so that $K \subset B(0,D) \smallsetminus B(0,d)$ and $A^j K \subset A^jB(0,D) \smallsetminus A^jB(0,d)$. The $A^j$ image of a ball $B(0,r)$ is contained in a (hyper)ellipsoid with small semi-minor axis $\frac{1}{C}{\ul{\lambda}}^j r$ and semi-major axis $C{\ol{\lambda}}^j r$. Thus a sufficient condition for $K \cap A^{j} K = \emptyset$ is $C \ol{\lambda}^jd > D$. If $d > 0$, that is $0 \notin K$, then there are at most $\log_{\ol{\lambda}}(\frac{D}{dC})$ terms in the sum that correspond to integrations over domains that superpose. Let then $S = \lceil \log_{\ol{\lambda}} (\frac{D}{Cd})\rceil$, so that
\[
\int_{K} \Psi_M(\xi) d\xi = \sum_{l = 1}^{S}\sum_{j = N(M)}^{+\infty} m(l+jS+N(M)) \int_{A^{l+jS+N(M)}K} |\wh{\psi}(\xi)|^2 d\xi .
\]
If $m$ is a constant, since the inner sum runs over disjoint domains, we have $\int_{K} \Psi_M(\xi) d\xi \leq m\, S\, \|\psi\|_{L^2(\R^n)}^2$, while in general we have boundedness if $\limsup_{j \to \infty} m(j) < \infty$.

We would like to observe explicitly that this last argument actually mimics the proof of Theorem \ref{theo:sufficient}, and can be found to be very close to the one used in \cite[Lemma 4.1]{ChristensenHasannasabLemvig}.

\subsubsection{Wavelets with composite dilations}\ \\
For $G = \R^n$ and $\Gamma$ a cocompact subgroup, the present setting allows us to consider $\{\alpha_h\}_{h \in H}$ to be any subset of $GL_n(\R)$ parametrized by $H$, with an appropriate measure $\varsigma$ which behaves well with respect to the constants (\ref{eq:metric}). This includes the case of so-called composite dilations \cite{GLLWW}, with no restriction on the discreteness of $H$.

A notable relevant case is provided by shearlets \cite{LabateShear}, \cite{KutyniokLabateShearlets}: let $G = \R^2$, let $\Gamma = \Z^2$, and let $H$ be any subset of $\R \times \R^+$. For $a \in \R^+$ and $s \in \R$, the shearlets automorphisms read $\alpha_{a,s}(x) = A_a^{-1}S_s^{-1}(x)$, with anisotropic dilations $A_a = \binom{a \ \ 0}{0 \, \sqrt{a}}$ and shears $S_s = \binom{1 \ s}{0 \ 1}$. 
Thus $\delta(a,s) = a^\frac32$, and
\[
\dalpha_{a,s}(\xi) = A_a\,{}^tS_s (\xi) = \left(
\begin{array}{cc}
a & 0\\
s\sqrt{a} & \sqrt{a}
\end{array}
\right)
\binom{\xi_1}{\xi_2} =
\binom{a \xi_1}{s\sqrt{a}\xi_1 + \sqrt{a}\xi_2} .
\]
The optimal bi-Lipschitz constants with respect to the Euclidean metric are given by the singular values of the matrix $A_a\,{}^tS_s$, i.e. by the square root of the eigenvalues of
\[
M(a,s) = \left(
\begin{array}{cc}
a & s\sqrt{a}\\
0 & \sqrt{a}
\end{array}
\right)
\left(
\begin{array}{cc}
a & 0\\
s\sqrt{a} & \sqrt{a}
\end{array}
\right)
= \left(
\begin{array}{cc}
a^2 + s^2 a & sa\\
sa & a
\end{array}
\right)
\]
so that
\begin{align*}
L(a,s) & = a^\frac12 \Big(a + s^2 + 1 + \sqrt{(a + s^2 + 1)^2 - 4 a}\Big)^\frac12\\
\ell(a,s) & = a^\frac12 \Big(a + s^2 + 1 - \sqrt{(a + s^2 + 1)^2 - 4 a}\Big)^\frac12 .
\end{align*}
Since, for any fixed $a$, as $s$ becomes large we can have arbitrarily large values of $L$ with arbitrarily small values of $\ell$, this family of automorphisms does not satisfy the expanding condition. Nevertheless, it is easy to see that it satisfies Property X. Indeed, if for simplicity we consider the invariant metric provided by the $\infty$ distance, whose balls are squares $B_\infty(0,\epsilon) = \{\xi \in \R^2 \, : \, \max\{|\xi_1|,|\xi_2|\} < \epsilon\}$, we can easily check (see e.g. Figure \ref{fig:shearlets}) that
\[
\sharp \left(\Z^2 \cap \dalpha_{a,s} B_\infty(0,\epsilon)\right) \leq (\lfloor 2\epsilon a\rfloor + 1)(\lfloor 2\epsilon \sqrt{a}\rfloor + 1)
\]
so, if we consider $\epsilon_0 = \frac12$, for all $\epsilon < \epsilon_0$ either we have $a \geq 1$ or we have only one point inside the ball. If $a \geq 1$ we then have
\[
\sharp \left(\Z^2 \cap \dalpha_{a,s} B_\infty(0,\epsilon)\right) \leq 1 + 2\epsilon \left(2\epsilon + \frac{1 + \sqrt{a}}{a}\right) a^{\frac32}
\]
where $\frac{1 + \sqrt{a}}{a} \leq 2$, so that we can choose $C_\epsilon = 4\epsilon(\epsilon+1)$.
\begin{figure}[h!]
\centering
\begin{tikzpicture}[scale=2.5]
 \draw[->] (-1,0) -- (1,0) coordinate (x axis);
 \draw[->] (0,-1) -- (0,1) coordinate (y axis);

 \draw[->] (1.8,0) -- (4.2,0) coordinate (x axis);
 \draw[->] (3,-1) -- (3,1) coordinate (y axis);

 \draw [thick] (-0.5,-0.5) rectangle (0.5,0.5);
 \node at (0.6,-0.1) {$\epsilon$};
 \node at (-0.1,0.6) {$\epsilon$};

 \draw [thick] (2,-0.9) -- (4,0.3) -- (4,0.9) -- (2,-0.3) -- (2,-0.9);
 \draw (4,0.05) -- (4,-0.05);
 \node at (4,-0.12) {$\epsilon a$};
 \draw (2.95,0.303) -- (3.05,0.303);
 \node at (2.72,0.32) {$\epsilon \sqrt{a}$};

 \foreach \x in {-3,...,3} \foreach \y in {-3,...,3}
 {
  \path[fill=red] (\x*0.26 + 3,\y*0.26) circle (0.02);
 }
\end{tikzpicture}
\caption{Left: the $\infty$ ball $B_\infty(0,\epsilon) = \{\xi \in \R^2 \, : \, \max\{|\xi_1|,|\xi_2|\} < \epsilon\}$. Right: a set $\dalpha_{a,s} B_\infty(0,\epsilon)$ and the integer lattice.}\label{fig:shearlets}
\end{figure}
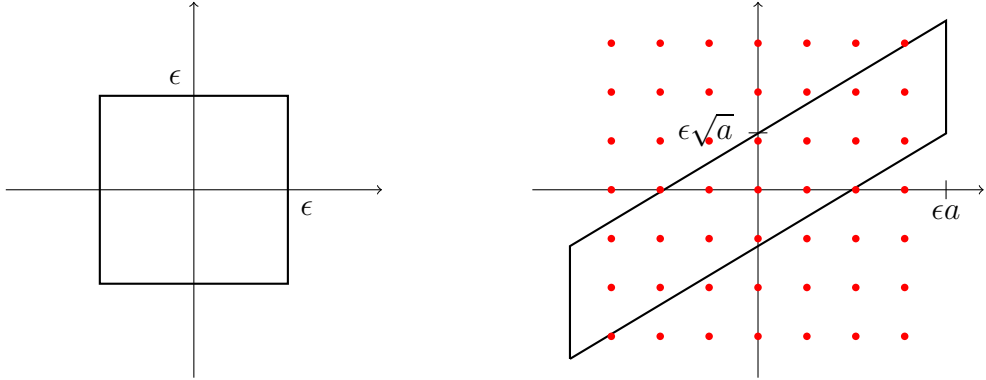

\section{Gabor-type frames}\label{sec:Gabor}

In this section we will show how the results of \S \ref{sec:boundedness} allow us to address the case of Gabor frames, even if modulations can not be directly realized in terms of spatial automorphisms. Nevertheless, by considering an appropriate Hilbert subspace and the set of automorphisms that define the Heisenberg group as a semidirect product, one can recover the discussed setting. We will then show that, even if these automorphisms are not expanding, they possess Property X as a simple consequence of Lemma \ref{lem:counting}.

\subsection{LCA-Gabor systems and automorphisms}

For simplicity, let us identify $\T = \{z \in \C \, : \, |z| = 1\}$ and use exponential coordinates on it. For any LCA group $\RR$, let $G = \RR \times \T$, with translations given by
\[
T(q,e^{2\pi i\varphi})f(x,e^{2\pi i\theta}) = f(x - q, e^{2\pi i(\theta - \varphi)}) \ , \quad f \in L^2(G)
\]
where we use the multiplicative notation for the composition on $\T$.
Let then $\PP \subset \wh{\RR}$, and let $\alpha : \PP \to \textnormal{Aut}(G)$ be given by
\[
\alpha_p(x,e^{2\pi i\theta}) = (x, e^{2\pi i(\theta + \langle p, x\rangle)})
\]
where we denote by $e^{2\pi i \langle p, x\rangle} \in \T$ the action of the characters of $\RR$. Note that this action is the one that defines the (reduced) Heisenberg group (see e.g. \cite[\S 4.1]{Thangavelu} or \cite[\S 1.3, \S 1.11]{Folland}) as the semidirect product $\wh{\RR} \ltimes \RR \times \T$. Note that, for this action, we have $\delta = 1$. Let also $\varsigma$ be a $\sigma$-finite Borel measure on $\PP$ such that this $\alpha_p$ is measurable.

\begin{definition}
For $\kappa \in \Z \smallsetminus \{0\}$, let $\pi_\kappa : \wh{\RR}\times\RR \to \mathcal{U}(L^2(\RR))$ be the projective representation
\[
\pi_\kappa(p,q)g(x) = e^{2\pi i \kappa\langle p, x\rangle} g(x - q)
\]
and let $\Lambda$ be a cocompact closed subgroup of $\RR$. We say that the Gabor system
\[
\mathbb{G}_\kappa(g,\Lambda,\PP) = \{\pi_\kappa(p,q)g \, : \, q \in \Lambda \, , \ p \in \PP\}
\]
is a $\varsigma$-frame of $L^2(\RR)$ with constants $0 < A \leq B < \infty$ if
\[
A \|u\|_{L^2(\RR)}^2 \leq \int_\PP \Big(\int_\Lambda |\langle u, \pi_\kappa(p,q)g\rangle_{L^2(\RR)}|^2  d\mu_\Lambda(q) \Big) d\varsigma(p) \leq B \|u\|_{L^2(\RR)}^2
\]
for all $u \in L^2(\RR)$.
\end{definition}

These systems were thoroughly studied in the recent work \cite{JakobsenLemvigGabor}.

\begin{definition}
For $\kappa \in \Z$, let $\Hil_\kappa$ be the Hilbert subspace of $L^2(G)$ given by
\[
\Hil_\kappa = \{f \in L^2(\RR\times\T) \, : \, f(x,e^{2\pi i\theta}) = u(x)e^{2\pi i \kappa\theta} \, , \ u \in L^2(\RR)\}
\]
and, for $f \in \Hil_\kappa$, let us call $u$ its $\RR$-representative.
\end{definition}

With this notation, we can can provide the main observation that allows us to relate Gabor frames to the affine frames of definition \ref{def:frames}.

\begin{lemma}\label{lem:gabortrick}
Let $\Lambda$ be a cocompact closed subgroup of $\RR$, and let $\Gamma = \Lambda \times \T$. 
Let $\psi \in \Hil_\kappa$, and let $g \in L^2(\RR)$ be its $\RR$-representative. Then the system $\mathcal{A}_{\Gamma,\PP}(\psi)$ is a $\varsigma$-frame of $\Hil_\kappa$ with constants $0 < A \leq B < \infty$ if and only if the Gabor system $\mathbb{G}_\kappa(g,\Lambda,\PP)$ is a $\varsigma$-frame of $L^2(\RR)$ with the same constants.
\end{lemma}
\begin{proof}
Let $f, \psi \in \Hil_\kappa$, and let $u, g \in L^2(\RR)$ be their $\RR$-representatives. Then
\begin{align*}
\langle f, D(p)T(q,e^{2\pi i \varphi})\psi\rangle_{L^2(G)} & = \int_{\T \times \RR} \!\!u(x)e^{2\pi i \kappa\theta} \ol{g(x - q)} e^{-2\pi i \kappa (\theta - \varphi + \langle p,x\rangle)} d\mu_\RR(x) d\theta\\
& = e^{2\pi i \kappa\varphi}\langle u, \pi_\kappa(p,q)g\rangle_{L^2(\RR)} .
\end{align*}
The proof follows by noting that $\|f\|_{L^2(G)} = \|u\|_{L^2(\RR)}$.
\end{proof}

\subsection{Property X and Calder\'on's bounds}
The dual action $\dalpha$ on $\wh{G} = \wh{\RR} \times \Z$ can be computed from the duality
\[
\langle (\xi,k), \alpha_p^{-1}(x,\theta)\rangle = e^{2\pi i (\langle \xi,x\rangle + k(\theta - \langle p,x\rangle))} \, ,
\]
where $(\xi,k) \in \wh{\RR} \times \Z$ and $(x,\theta) \in \RR \times \T$, so that
\begin{equation}\label{eq:Gaborautomorphisms}
\dalpha_p (\xi,k) = (\xi - kp , k)
\end{equation}
where by $k p$ we mean the iterated $\wh{\RR}$ composition $\underbrace{p + p + \dots + p}_k$ if $k > 0$. If $k < 0$, compose $|k|$ times $-p$, while if $k = 0$ that is the neutral element of $\wh{R}$.

Given an invariant distance $\dr$ on $\wh{\RR}$ satisfying (\ref{eq:Vitali}), the distance
\[
\dg((\xi,k),(\eta,l)) = \dr(\xi,\eta) + |k - l|
\]
is invariant on $\wh{G}$ and satisfies (\ref{eq:Vitali}). For $\xi \in \wh\RR$, let us also denote by \[\|\xi\| = \dr(\xi,\id_{\wh{\RR}}).\]

\begin{lemma}
The map $\dalpha_p$ is bi-Lipschitz with optimal constants
\[
L(p) = 1 + \|p\| \ , \quad \ell(p) = \frac{1}{1 + \|p\|} = \frac{1}{L(p)} .
\]
\end{lemma}
\begin{proof}
By definition, we have $L(p) = \displaystyle\sup_{(\xi,k) \neq 0} \frac{\dr(\xi,kp) + |k|}{\|\xi\| + |k|}$, and
\[
\frac{\dr(\xi,kp) + |k|}{\|\xi\| + |k|} \leq \frac{\|\xi\| + |k|\,\|p\| + |k|}{\|\xi\| + |k|} = 1 + \frac{\|p\|}{1 + \|\xi\|/|k|} \leq 1 + \|p\| \, ,
\]
and the $\sup$ is a $\max$ attained at $\xi = \id_{\wh{\RR}}$.

On the other hand, for $\ell(p) = \displaystyle\inf_{(\xi,k) \neq 0} \frac{\dr(\xi,kp) + |k|}{\|\xi\| + |k|}$ we have
\begin{displaymath}
\frac{\dr(\xi,kp) + |k|}{\|\xi\| + |k|} \geq \frac{\|\xi - kp\| + |k|}{\|\xi - kp\| + |k|\,\|p\| + |k|} = \frac{1}{1 + \frac{|k|}{\|\xi - kp\| + |k|}\,\|p\|} \geq \frac{1}{1 + \|p\|} \, ,
\end{displaymath}
and the $\inf$ is a $\min$ attained at any $\xi = kp$.
\end{proof}

Thus, in order to comply with assumption {\bf III)} of \S\ref{sec:boundedness} and make use of the results of that section, we need to require that the function $\|\cdot\|$ be $\varsigma$-measurable on $\PP$.

\begin{remark}\label{rem:nonexpandingGabor}
Since in this case $\ell = \frac{1}{L}$ is optimal, then these automorphisms are not expanding. Indeed
\[
\ell(p) < M \iff L(p) > \displaystyle\frac{1}{M} .
\]
\end{remark}

Although these automorphisms do not satisfy the expanding property, we shall prove that they satisfy Property X, and consequently the results of \S\ref{sec:boundedness} can be applied. This is proved in the following two lemmata.

\begin{lemma}\label{lem:GaborL}
For all $\epsilon < 1$, we have
\begin{equation}\label{eq:characteristic}
\one{B((\xi_0,k_0),\epsilon)}(\xi,k) = \one{B_{\wh{\RR}}(\xi_0,\epsilon)}(\xi) \delta_{k,k_0} .
\end{equation}
In particular, we have
\[
\wh{\mathcal{G}}^{\,\Hil_\kappa}_{1} = \{(\xi,k)\in \wh{G} \, : \, \xi \in \wh{\RR}\, , \ k = \kappa\}
\]
which is isomorphic to $\wh{\RR}$ for any $\kappa \in \Z$.
\end{lemma}
\begin{proof}
A $\wh{G}$-metric ball of radius $\epsilon$ and center $(\xi_0,k_0) \in \wh{G}$ is defined by
\[
\dr(\xi,\xi_0) + |k - k_0| < \epsilon
\]
which reduces, for $\epsilon < 1$, to
\[
B((\xi_0,k_0),\epsilon) = \{\xi \in B_{\wh{\RR}}(\xi_0,\epsilon) \, , \ k = k_0 \}
\]
hence proving (\ref{eq:characteristic}). As a consequence, for $\epsilon < 1$, if $\wh{f} = \one{B((\xi_0,k_0),\epsilon)}$ then $f$ belongs to $\Hil_\kappa$ if and only if $k_0 = \kappa$.
\end{proof}
\begin{lemma}\label{lem:GaborX}
The automorphisms (\ref{eq:Gaborautomorphisms}) have Property \textnormal{X}.
\end{lemma}
\begin{proof}
Since for any $\epsilon < 1$ we have
\[
\dalpha_p B((\xi_0,k_0),\epsilon) = B((\xi_0-k_0 p,k_0),\epsilon) \, ,
\]
then $\dalpha_p B(\id,\epsilon) = B(\id,\epsilon)$. Thus, since $\id = (\id_{\wh{\RR}},0)$, we get
\[
\sharp (\Gamma^\perp \cap \dalpha_p B(\id,\epsilon)) = \sharp \{\Gamma^\perp \cap B(\id,\epsilon)\} .
\]
Since this does not depend on $p$, and since in this situation $\delta = 1$, we then get the estimate (\ref{eq:PropertyX}) for all $p \in \PP$, for all radii $r < 1$.
\end{proof}

This in turn provides a proof of the following result for Gabor-type frames, which was recently proved in \cite[Corollary 5.6]{JakobsenLemvigGabor} with completely different techniques, and generalizes several previous statements such as \cite[Prop. 4.1.4]{HeilWalnut1989}. Without loss of generality, we have fixed $\kappa = 1$.
\begin{theorem}\label{theo:GaborCalderon}
For $g \in L^2(\RR)$, let $\{\pi(p,q)g\,:\, p \in \PP \, , \ q \in \Lambda\}$ be a $\varsigma$-frame of $L^2(\RR)$ with constants 
$0 < A \leq B < \infty$. Then
\[
A \leq \int_\PP |\wh{g}(\xi - p)|^2 d\varsigma(p) \leq B \quad \textnormal{a.e.}\ \xi \in \wh{\RR} .
\]
\end{theorem}
\begin{proof}
By Lemma \ref{lem:gabortrick}, the hypothesis of having a Gabor frame of $L^2(\RR)$ is equivalent to having an affine frame of $\Hil_1$. Now, since $\delta = 1$, then the quantity we have called $\Psi_M$ coincides with the Calder\'on's sum, so its local integrability is a direct consequence of the Bessel inequality, and hence it is always verified for frames. Lemma \ref{lem:GaborX} allows us to use Theorem \ref{theo:main}, so the conclusion follows by Lemma \ref{lem:GaborL}.
\end{proof}

\newpage
\appendix

\section{More on the Expanding Property}\label{sec:more}

We discuss here some related issues concerning expanding automorphisms and the optimality of Lemma \ref{lem:counting}.

\subsection*{On the optimality of Lemma \ref{lem:counting}}

Let us observe first that, for any fixed $\dalpha$, we can find a small enough $r_0$ such that $\nu_{\wh{G}}(\Omega_{\dalpha}^r) = \nu_{\wh{G}}(\dalpha B(\id,r))$. Thus, by Lemma \ref{lem:counting} and Remark \ref{rem:measureballdilations}
\[
\sharp (\Gamma^\perp \cap \dalpha B(\id,r))\leq \frac{\nu_{\wh{G}}(B(\id,2r))}{\nu_{\wh{G}}(B(\id,r))} \quad \forall r < r_0 .
\]
Note that the ratio on the right hand side is always larger than 1, but, by (\ref{eq:Vitali}), it will be bounded by the doubling constant.
On the other hand, when $\Gamma^\perp$ is a uniform lattice, $\Omega$ can be chosen to be compact and in this case, for any fixed $\dalpha$, we can find a large enough $r_1$ such that $\Omega_{\dalpha}^r = \Omega$ for all $r > r_1$. Thus, by Lemma \ref{lem:counting},
\[
\sharp (\Gamma^\perp \cap \dalpha B(\id,r))\leq c \nu_{\wh{G}}(\dalpha B(\id,2r)) \quad \forall r > r_1
\]
where $c = \frac{1}{\nu_{\wh{G}}(\Omega)}$.
We now prove that a lower estimate also holds.
\begin{lemma}\label{lem:optimal}
Let $\wh{G}$ be a second countable LCA group with Haar measure $\nu_{\wh{G}}$, let $\Gamma^\perp$ be a discrete subgroup of $\wh{G}$, and let $\dalpha \in \textnormal{Aut}(\wh{G})$. Let also $\dg : \wh{G} \times \wh{G} \to \R^+$ be an invariant metric, and let
$B(\xi_0,r) = \{\xi \in \wh{G} \, : \, \dg(\xi,\xi_0)<r\}$ for $\xi_0 \in \wh{G}, r > 0$.
Then
\[
\sharp \Big(\Gamma^\perp \cap \dalpha B(\id,2r)\Big) \geq \frac{1}{\nu_{\wh{G}}(\Omega_{\dalpha}^r)}\nu_{\wh{G}}(\dalpha B(\id,r)) \quad \forall \ r > 0
\]
where $\Omega_{\dalpha}^r$ is as in (\ref{eq:neighborhood}) for any $\Omega \subset \wh{G}$ a $\nu_{\wh{G}}$-measurable section of $\wh{G}/\Gamma^\perp$.
\end{lemma}

\begin{proof}
By the same argument that leads to (\ref{eq:basiccounting}) we can also get
\[
\nu_{\wh{G}}(\dalpha(B(\id,r))) = \int_{\Omega_{\dalpha}^r} S_{\dalpha}^r(\xi) d\nu_{\wh{G}}(\xi) \leq \nu_{\wh{G}}(\Omega_{\dalpha}^r) \sup_{\xi \in \Omega_{\dalpha}^r} S_{\dalpha}^r(\xi) .
\]
The proof can then be concluded by showing that
\[
S_{\dalpha}^r(\xi) \leq S_{\dalpha}^{2r}(\id) \quad \forall \ \xi \in \Omega_{\dalpha}^{r}.
\]
To see this, let $\lambda_{\ol{\xi}}$ be such that $\ol{\xi} \in \dalpha B(\id,r) + \lambda_{\ol{\xi}}$. Then
\[
\big(\dalpha B(\id,r) - \ol{\xi}\big) \subset \big(\dalpha B(\id,2r) - \lambda_{\ol{\xi}}\big)
\]
because, for any $\zeta = \dalpha(\eta) \in \big(\dalpha B(\id,r) - \ol{\xi}\big)$ we have
\[
\dg(\eta,-\dalpha^{-1}(\lambda_{\ol{\xi}})) \leq \dg(\eta,-\dalpha^{-1}(\ol{\xi})) + \dg(\dalpha^{-1}(\ol{\xi}),\dalpha^{-1}(\lambda_{\ol{\xi}})) < 2r.
\]
The proof then follows because for all $\lambda \in \Gamma^\perp$ we have $S_{\dalpha}^{2r}(\lambda) = S_{\dalpha}^{2r}(\id)$.
\end{proof}

\subsection*{On the definition of expanding on a subspace}
In \cite{HLW}, with the correction discussed in \cite{GuoLabate}, a definition of a matrix in $\R^n$ that is expanding on a subspace is given, that allows to obtain a lattice counting estimate given by \cite[Lemma 5.11]{HLW} and \cite[Lemma 3.3]{GuoLabate}, discussed later on.\\
Their notion of expansiveness is as follows: given a non-zero linear subspace $F \subset \R^n$, a matrix $A \in GL_n(\R)$ is expanding on $F$ if there exists a linear subspace $E \subset \R^n$ such that
\begin{itemize}
\item[i.] $\R^n = F + E$ and $F \cap E = \{0\}$
\item[ii.] $A(F) = F$ and $A(E) = E$
\item[iii.] $\exists \ 0 < k \leq 1 < \gamma < \infty$ such that
$|A^j x| \geq k \gamma^j |x|$ \ $\forall j \geq 0$, $\forall x \in F$
\item[iv.] $\exists \ a > 0$ such that $|A^j x| \geq a |x|$ \ $\forall j \geq 0$, $\forall x \in E$ .
\end{itemize}
The classical definition of expanding matrix can be obtained in the special case of $E = \{0\}$, that is $F = \R^n$, and reduces to iii. or, equivalently, to saying that all eigenvalues $\lambda$ of $A$ are such that $ |\lambda|> 1$ (see also \S \ref{sec:examples}).

\begin{proposition}\label{prop:expandingsubspace}
Let $A \in GL_n(\R)$ be expanding on a subspace. Then it is expanding in the sense of Definition \ref{def:expanding}.
\end{proposition}
\begin{proof}
Since, by i., any $x \in \R^n$ can be written as $x = x_F + x_E$ with $x_F \in F$ and $x_E \in E$, points iii. and iv. imply that
\[
|A^j x| \geq k \gamma^j |x_F| + a |x_E| \geq \min\{k \gamma^j , a\} |x| \quad \forall j \geq 0
\]
for all $x \in \R^n$. Let us call $\ell(j) = \min\{k \gamma^j , a\}$. By \cite[Lemma 5.1]{HLW}, we also have that there exists $\beta > \|A\|$ such that, calling $L(j) = \beta^j$, we get
\[
\ell(j)|x| \leq |A^j x| \leq L(j) |x| \quad \forall x \in \R^n .
\]
Assume now, by contradiction, that for all $M_0, N_0 > 0$ there exists an $M > M_0$ such that $L(j) > M$ while $\ell(j) < N_0$. For a large value of $M_0$, this means that $j$ must be a large positive integer. But $\ell(j)$ can not be smaller than $a$, so any $N_0 < a$ provides a contradiction.
\end{proof}

We observe that the converse of Proposition \ref{prop:expandingsubspace} does not hold because, for example, a collection of rotations satisfies Definition \ref{def:expanding} but is not expanding on a subspace.

We also remark that in the present work we are not restricting ourselves to uniform (full-rank) lattices, but we consider the larger class of annihilators of cocompact subgroups (see also \cite{BownikRoss} for a thorough discussion of this point).


\subsection*{On other notions of expansiveness on LCA groups}
A lattice counting estimate similar to Property X is obtained in LCA groups, for uniform lattices and expanding automorphisms, in \cite[Lemma 4.11]{KutyniokLabate}. There, an automorphism $A \in \textnormal{Aut}(\wh{G})$ on an LCA group $\wh{G}$ is called expanding if
\[
\dg(A(\xi),\id) \geq c \,\dg(\xi,\id) \quad \forall \ \xi \neq \id \ , \ \ \textnormal{for some} \ c > 1.
\]
This generalizes classical expanding matrices, and it immediately implies that the family $\{A^j\}_{j \in \Z}$ is uniformly expanding according to Definition \ref{def:expanding}.

\vspace{1.5ex}
\noindent
We finally note that a notion of expansiveness also plays an important role in topological dynamical systems and ergodic theory, see e.g. \cite{KitchensSchmidt}. In this context, an action $\dalpha : H \to \textrm{Aut}(\wh{G})$ of a countable group $H$ on a locally compact group $\wh{G}$ is said to be expanding if there exists a neighborhood $U$ of $\id$ such that $\bigcap_{h \in H} \dalpha_h(U) = \{\id\}$.
By \cite[Theorem 7.3]{KitchensSchmidt}, if $\wh{G}$ is compact and connected, and it admits an expansive action for $H$ abelian and finitely generated, then $\wh{G}$ is abelian. However, in this setting, this notion of expansiveness is different from that of Definition \ref{def:expanding}, since it includes cases that we exclude such as that of Example \ref{ex:BAD}.

\newpage

\section{Proof of Lemma \ref{lem:Weil}}\label{sec:Weil}
For the sake of completeness, we provide here a proof of Lemma \ref{lem:Weil}.
\begin{proof}
Point $i.$ is a direct consequence of the definition of fundamental set, since $\displaystyle\nu_{\wh{G}}\Big(\wh{G} \smallsetminus \bigsqcup_{\lambda \in \Gamma^\perp} (\Omega + \lambda)\Big) = 0$. The issue of proving point $ii.$ consists of showing that the measures $\mu_\Gamma$ and $\nu_{\wh{G}}$ make the $\Gamma$-Fourier transform a unitary map from $L^2(\Gamma)$ to $L^2(\Omega)$. To see this, let us choose a Haar measure $\wh\theta$ on $\Omega$ such that the $\Gamma$-Fourier transform is unitary with $\mu_\Gamma$ on $\Gamma$, and let us choose a Haar measure $\wh{\theta}_0$ on $\Gamma^\perp$ such that the $\Gamma^\perp$-Fourier transform is unitary with $\kappa$ on $G/\Gamma$. By uniqueness of the Haar measure, there exist two positive constants $c$ and $c_0$ such that
\[
\wh{\theta} = c \, \cdot \, \nu_{\wh{G}} \ , \quad \wh{\theta}_0 = c_0 \, \cdot \, \textnormal{counting measure} .
\]
We only need to prove that $c = 1$. By \cite[(31.46), (c)]{HewittRossII}, we then have that
\[
\int_{\wh{G}} \phi(\xi) d\nu_{\wh{G}}(\xi) = c\,c_0\,\sum_{\lambda \in \Gamma^\perp} \int_\Omega \phi(\xi + \lambda) d\nu_{\wh{G}}(\xi) \qquad \forall \ \phi \in L^1(\wh{G}) .
\]
Now, since $i.$ holds, we get $c\,c_0 = 1$. Moreover, by definition of $\wh{\theta_0}$ we have
\[
\int_{G/\Gamma} |f(y)|^2 d\kappa(y) = c_0 \sum_{\lambda \in \Gamma^\perp} |\int_{G/\Gamma} f(y) \langle \lambda, y\rangle d\kappa|^2 .
\]
In the above identity we can choose $f \equiv 1$ because $G/\Gamma$ is compact, hence obtaining, by \cite[Lemma (23.19)]{HewittRossI},
\[
\kappa(G/\Gamma) = c_0 \sum_{\lambda \in \Gamma^\perp} |\int_{G/\Gamma} \langle \lambda, y\rangle d\kappa|^2 = c_0 \kappa(G/\Gamma)^2 .
\]
Thus, $c_0 = \kappa(G/\Gamma) = 1$, and since $c c_0 = 1$ we deduce $c = 1$ as wanted.
\end{proof}

\vspace{1ex}
\noindent
D. Barbieri, Universidad Aut\'onoma de Madrid, 28049 Madrid, Spain.\\
davide.barbieri@uam.es

\vspace{1ex}
\noindent
E. Hern\'andez, Universidad Aut\'onoma de Madrid, 28049 Madrid, Spain.\\
eugenio.hernandez@uam.es

\vspace{1ex}
\noindent
A. Mayeli, City University of New York, Queensborough and the Graduate Center, New York, USA.\\
amayeli@gc.cuny.edu

\end{document}